\documentclass[leqno]{amsart}

\usepackage[centertags]{amsmath}
\usepackage{mathabx}
\usepackage{amsfonts}
\usepackage{amssymb}
\usepackage{amsthm}
\usepackage[T1]{fontenc}
\usepackage{newlfont}
\usepackage{color}
\usepackage{enumerate}

\newtheorem{theorem}{Theorem}
\newtheorem{fact}{Fact}
\newtheorem{lemma}{Lemma}
\newtheorem{corollary}{Corollary}

\theoremstyle{definition}

\newtheorem{example}{Example}
\newtheorem{remark}{Remark}
\newtheorem{problem}{Problem}

\newcommand{\ord}{{\mathrm{ ord}}\,}
\newcommand{\modd}{\,{\mathrm{ mod}}\,}

\newcommand{\dd}{{\mathrm{ d}}}
\newcommand{\id}{{\mathrm{ id}}}

\newcommand{\ps}[2]{\boldsymbol{#1}[\hspace{-0.05cm}[{#2}]\hspace{-0.05cm}]}
\newcommand{\pol}[2]{\boldsymbol{#1}[{#2}]}
\newcommand{\G}[1]{\mbox{\gothic{#1}}}
\newfont{\gothic}{eufm9 scaled 1100}
\begin{document}
\title{An algebraic structure of groups of~solutions of~the third Acz{\'e}l-Jabotinsky formal differential equation}
\author[W. Jab{\l}o{\'n}ski]{Wojciech Jab{\l}o{\'n}ski}
\address{Faculty of Applied Mathematics, AGH University of Science and Technology, Mickiewicza 30, 30-059 Krak{\'o}w, POLAND}
\email{wjablonski@agh.edu.pl}
\author[L. Reich]{Ludwig Reich}
\address{Institute of Mathematics, Karl-Franzens-University Graz,
                Heinrichstrasse 36, A-8010 Graz, AUSTRIA}
\email{ludwig.reich@kfunigraz.ac.at} 
\date{}
\begin{abstract}
We study the algebraic structure of the groups of solutions of the third Acz{\'e}l-Jabotinsky differential equation
$$
(H\circ \Phi)(x)=\frac{\dd\Phi}{\dd x}\cdot H(x)
$$
in the rings of formal power series and truncated formal power series $\ps{k}{x}_s$, where $s$ is either a positive integer or $s=\infty$ and $H\in\ps{\boldsymbol{k}}{x}_s$ is given. We give also the detailed description of all  solutions of this equation.
\end{abstract}

\maketitle

\pagestyle{myheadings} \markboth{W. Jab{\l}o{\'n}ski and L.
Reich}{Iterations groups in formal power series and ... }


\section{Introduction}

Let $\boldsymbol{k}$ be a field of characteristic~$0$ with the unit~$1$ and let $\boldsymbol{q}\subset\boldsymbol{k}$ be a simple field (which is clearly isomorphic to the field of all rationals). The problem of determining of one-parameter iteration groups $(F_t)_{t\in \boldsymbol{k}}$ of invertible formal power series over~$\boldsymbol{k}$ leads to the famous translation equation
\begin{equation}\label{tre}
\left\{\begin{array}{l}
F_{t_1+t_2}=F_{t_1}\circ F_{t_2}\qquad\mbox{ for }\,t_1,t_2\in \boldsymbol{k},\\
F_1=\id.
\end{array}\right.
\end{equation}
Since $(\boldsymbol{k},+)$ is commutative, so any itaration group $(F_t)_{t\in \boldsymbol{k}}$ is a family of commuting formal power series (see~\cite{Reich2}), that is it satisfies
\begin{equation}\label{comm}
F_{t_1}\circ F_{t_2}=F_{t_2}\circ F_{t_1}\qquad\mbox{ for }\,t_1,t_2\in \boldsymbol{k}.
\end{equation}
If $F_t$ is in some sense differentiable with respect to the time-variable~$t$, then differentiating~\eqref{comm} with respect to~$t_1$ we get
$$
\frac{\partial F_{t_1}}{\partial t}(F_{t_2}(x))=\frac{\partial F_{t_2}}{\partial x}(F_{t_1}(x))\cdot\frac{\partial F_{t_1}}{\partial t}(x).
$$
Putting $t_1=0$ and next $t_2=1$, with $H:=\frac{\partial F_0}{\partial t}$ and $\Phi:=F_1$, we obtain the third Acz{\'e}l-Jabotinsky differential equation
$$
(H\circ\Phi)(x)=\frac{\dd\Phi}{\dd x}(x)\cdot H(x).
\leqno{({\mathrm{AJ}}_{\infty}(H))}
$$
The group $\Gamma^{\infty}\subset\ps{k}{x}$ of invertible formal power series is the projective limit of the sequence $(\Gamma^s)_{s\in\mathbb{N}}$ of groups $\Gamma^s\subset\ps{k}{x}_s$ of invertible truncated formal power series. One can therefore ask for the form of one-parameter iteration groups $(F^{[s]}_t)_{t\in G}$ invertible truncated formal power series, and also here we can derive the Acz{\'e}l-Jabotinsky differential equation
$$
(H\circ\Phi)(x)=\frac{\dd\Phi}{\dd x}(x)\cdot H(x) \quad \mod x^{s+1}\leqno{({\mathrm{AJ}}_s(H))}
$$
in the ring $\ps{k}{x}_s$ of $s$-truncated formal power series. We denote by ${\mathcal{S}}({\mathrm{AJ}}_s(H))$ the set of all solutions of~$({\mathrm{AJ}}_s(H))$, where $s\in\mathbb{N}\cup\{\infty\}$. Put $\ps{k}{x}_{\infty}:=\ps{k}{x}$.

For a fixed positive integer $n$ by $\boldsymbol{E}_n\subset\boldsymbol{k}^*:=\boldsymbol{k}\setminus\{0\}$ we denote the set of all roots of the unit~$1\in\boldsymbol{k}$, that is the set of all roots of the polynomial $x^n-1\in\boldsymbol{k}[x]$.

Let us fix $s\in\mathbb{N}\cup\{\infty\}$ and $H(x)=\sum_{j=l+1}^sh_jX^j\in\ps{k}{x}_s$, where $l\in\mathbb{N}\cup\{0\}$ is a non-negative integer and $h_{l+1}\neq0$. It is a simple observation that $\Phi(x)$ is a solution of $({\mathrm{AJ}}_s(H))$ if and only if $\Phi(x)$ is a solution of~$({\mathrm{AJ}}_s(\frac{1}{h_{l+1}}H))$. Hence, without loss of generality we assume that $H(x)=x^{l+1}+\sum_{j=l+2}^s h_jX^j$. We begin with general properties of the set of solutions~${\mathcal{S}}({\mathrm{AJ}}_s(H))$. We prove that ${\mathcal{S}}({\mathrm{AJ}}_s(H))$ is a group with respect to the substitution. Further on, we solve~$({\mathrm{AJ}}_s(H))$ considering two general cases:
\begin{description}
\item[case 1] $l=1$,
\item[case 2] $l\geq2$.
\end{description}
In the first case we give two representations of solutions $\Phi(x)=c_1x+\sum_{j=2}^sc_jx^j\in{\mathcal{S}}({\mathrm{AJ}}_s(H))$. These representations allow us to prove that the group ${\mathcal{S}}({\mathrm{AJ}}_s(H))$ is commutative and isomorphic to the multiplicative group $(\boldsymbol{k}^*,\cdot)$. The second case is much more complicated. Firstly, for $s=\infty$ we have two parameter representation of the group~${\mathcal{S}}({\mathrm{AJ}}_{\infty}(H))$. Also here the group of solutions ${\mathcal{S}}({\mathrm{AJ}}_{\infty}(H))$ is commutative.
For $s\in\mathbb{N}$ solutions of~${\mathrm{AJ}}_s(H))$ depend on some additional free parameters. Contrary to previous cases, the group ${\mathcal{S}}({\mathrm{AJ}}_s(H))$ is not commutative for $l\geq3$.

\section{The rings of formal power series and truncated formal power series}
\label{ringfps}

In the ring $\ps{k}{x}$ of formal power series $\sum_{j=0}^{\infty}c_jx^j$ over $\boldsymbol{k}$ we define the order of any formal power series by
$\ord\left(\sum_{j=0}^{\infty}c_jx^j\right)= \min\{j\in\mathbb{Z}:j\geq0, c_j\neq0\,\}$, with the assumption $\min\emptyset=\infty$. In the set $x\ps{k}{x}$ of formal power series $f$ with $\ord f\geq1$ we define a substitution as follows, if $f(x)=\sum_{j=1}^{\infty}c_jx^j, g(x)=\sum_{j=1}^{\infty}d_jx^j\in\ps{k}{x}$ then
$$
(f\circ g)(x)=\sum_{j=1}^{\infty}c_j\left(\sum_{l=1}^{\infty}d_lx^l\right)^j.
$$
Then the set $\Gamma^{\infty}=\{f(X)\in\ps{k}{X}:\ord{f(X)}=1\,\}$ is a group under substitution $\circ$ with the unit $L_1(x)=x=\id$.

Fix a positive integer $s\in\mathbb{N}$. A ring $\ps{k}{x}_s$ of $s$-truncated formal power series is the quotient ring $\ps{k}{X}/\G{m}^{s+1}$ where $\G{m}=(x)=x\ps{k}{x}=\{f\in\ps{k}{x}:\ord f \geq 1\,\}$ and $\G{m}^{s+1}=(x)^{1+1}=x^{s+1}\ps{k}{x}=\{f\in\ps{k}{x}:\ord f \geq s+1\,\}$. To each coset $f+\G{m}^{s+1}$ with $f(x)=\sum_{j=0}^{\infty}c_jX^j\in\ps{k}{x}$ we associate the $s$-truncation of $f$ given by
$$
f^{[s]}(x):=\sum_{j=0}^sc_jx^j\in\ps{k}{x}_s\subset
\pol{k}{x}\subset\ps{k}{x}.
$$
In so defined $\ps{k}{x}_s$ we introduce operations of addition, multiplication and substitution in the following way: for $f_1,f_2\in\ps{k}{x}_s$ we put $(f_1+f_2)(x):=f_1(x)+f_2(x)$, $(f_1\cdot f_2)(x):=(f_1\cdot f_2)^{[s]}(x)$ and $(f_1\circ f_2)(x):=(f_1\circ f_2)^{[s]}(x)$. Also here the set $\Gamma^s=\{f\in\ps{k}{x}_s:\ord f=1\,\}$ is a group under substitution and with the unit $L_1$.

As one can see, the ring $\ps{k}{x}_{\infty}:=\ps{k}{x}$ is a projective limit of the sequence of the rings $(\ps{k}{x}_s)_{s\in\mathbb{N}}$. Projections $\pi^r_s:\ps{k}{x}_r\to\ps{k}{x}_s$ given by
$$
\pi^r_s\left(\sum_{j=0}^rc_jx^j\right)=\sum_{j=0}^sc_jx^j
$$
for $r\in\mathbb{N}\cup\{\infty\}$ and $s\in\mathbb{N}$, are cononical rings epimorphisms and they satisfy $\pi^p_q\circ\pi^q_r=\pi^p_r$ for $p\geq q\geq r$. Similarly, the group $\Gamma^{\infty}$ is a projective limit of the sequence $(\Gamma^s)_{s\in\mathbb{N}}$ with cononical groups epimorphisms $\pi^r_s:\Gamma^r\to\Gamma^s$ for $r\in\mathbb{N}\cup\{\infty\}$ and $s\in\mathbb{N}$.

Let us fix $s\in\mathbb{N}\cup\{\infty\}$. The invert and the $n$-th power of $f\in\Gamma^s$ with respect to the substitution we denote by $f^{\circ-1}$ and $f^{\circ n}$, respectively. By a semicanonical form of order $l\in\mathbb{N}$ of a~formal power series in $\Gamma^s$ we mean every $f\in\Gamma^s$ such that $f(x)=\sum_{j=0}^rc_{jl+1}x^{jl+1}$, where $r$ is either the last positive integer with $rl+1\leq s$ if $s<\infty$ or $r=\infty$, otherwise. Denote by ${\mathcal{N}}^{s,l}$ the family of all such $f\in\Gamma^s$. If $c\in\boldsymbol{E}_l$ is primitive of order $l$ then
$$
{\mathcal{N}}^{s,l}=\left\{f\in\Gamma^s:f\circ L_c=L_c\circ f\,\right\},
$$
thus ${\mathcal{N}}_l^s$ is a subgroup of $\Gamma^s$. Let $L_c(x)=cx$ for $c\in\boldsymbol{k}^\star$ dente the linear form. For fixed $l\in\mathbb{N}$ we put
$$
{\mathcal{L}}=\{L_c: c\in\boldsymbol{k}^\star\,\}\quad\mbox{ and }\quad {\mathcal{L}}_l=\{L_c\in{\mathcal{L}}: c\in\boldsymbol{E}_l\,\}.
$$
Clearly ${\mathcal{L}}_l\subset{\mathcal{L}}\subset\Gamma^s$ for every $s\in\mathbb{N}\cup\{\infty\}$ and they are subgroups of $\Gamma^s$. Let us consider the following subgroups of $\Gamma^s$:
$$
\Gamma_1^s=\ker\pi^s_1=\{f\in \Gamma^s: f(x)=x\mod x^2\,\},\;
{\mathcal{N}}^{s,l}_l=\{f\in\mathcal{N}^{s,l}: \pi^s_1f\in{\mathcal{L}}_l\,\}.
$$
Clearly $\Gamma^s_1$ is normal in $\Gamma^s$, but ${\mathcal{L}}_m$ and ${\mathcal{L}}$ are not normal for $s\geq2$. Furthermore, ${\mathcal{L}}\cap\Gamma^s_1=\{L_1\}$ and if $f\in\Gamma^s$ with $\pi^s_1f=L_c$ for some $c\in\boldsymbol{k}^*$, then $f=L_c\circ(L_{c^{-1}}\circ f)$ and $L_{c^{-1}}\circ f\in\Gamma^s_1$, so ${\mathcal{L}}\circ\Gamma^s_1:=\{L\circ f:L\in\mathcal{L}^s, f\in\Gamma^s_1\,\}=\Gamma^s$, hence $\Gamma^s={\mathcal{L}}\overline{\circledcirc}\Gamma^s_1$ is a~semidirect product of the subgroup ${\mathcal{L}}$ and the normal subgroup
$\Gamma^s_1$. On the other hand, both ${\mathcal{L}}_l$ and ${\mathcal{N}}^{s,l}_1=\ker\pi^s_1|_{\mathcal{N}^{s,l}_l} \subset{\mathcal{N}}^{s,l}$ are normal divisors in ${\mathcal{N}}^{s,l}_m$,  so ${\mathcal{N}}^{s,l}_l={\mathcal{L}}_m\circledcirc {\mathcal{N}}^{s,l}_1$ is a direct product of the subgroups ${\mathcal{L}}_l$ and ${\mathcal{N}}^{s,l}_1$.

\section{A Description and properties of the substitution in~$\ps{k}{x}$ and $\ps{k}{x}_s$}
\label{propert}

In order to prove the the form of solution of the Acz{\'e}l-Jabotinsky formal differential equation we have to justify a detailed description of the substitution law in the rings~$\ps{k}{x}$ as well as~$\ps{k}{x}_s$. Any positive integer $n$ we can treate as $n\cdot1\in\boldsymbol{q}\subset\boldsymbol{k}$. For fixed integers $i$ and $j$ let $|i,j|$ be a set of all integers $n$ such that $i\leq n\leq j$, and let $|i,\infty|$ denote the set of all integers $n\geq i$. We assume additionally that $\sum_{t\in\emptyset}a_t=0$ and $\prod_{t\in\emptyset}a_t=1$. The formula which describes the substitution~$\circ$ follows directly form the following lemma.

\begin{lemma}[multinomial theorem]\label{f3}
Let $l,m$ be positive integers and let $x_1,\ldots,x_m$ be independent indeterminates. Then
$$
\left(x_1+x_2+\ldots+x_m\right)^l=
\sum_{\substack{\overline{u}_m=(u_1,\ldots, u_m)\in|0,l|^m,\\ u_1+\ldots +u_m=l}}B_{\overline{u}_m}^l
\prod_{j=1}^mx_j^{u_j},
$$
where $B_{\overline{u}_m}^l=\frac{l!}{\prod\limits_{j=1}^m u_j!}\in\boldsymbol{n}:=\{n\cdot1\in\boldsymbol{k}:n\in\mathbb{Z},\,n\geq1\,\}$.
\end{lemma}

Applying Lemma~\ref{f3} it is easily to check that if $\sum_{i=1}^{\infty}a_ix^i,\sum_{i=1}^{\infty}b_i
x^i\in\ps{k}{x}$, then
$$
\sum_{i=1}^{\infty}a_i\left(\sum_{j=1}^{\infty}b_jx^j\right)^i=
\sum_{n=1}^{\infty}d_nx^n,
$$
if and only if
\begin{equation}
\label{7}
d_n=\sum_{i=1}^na_i\sum_{\overline{u}_n \in U_{n,i}}
B_{\overline{u}_n} \prod_{j=1}^n b_j^{u_j} \qquad \mbox{ for }\,n\in|1,\infty|,
\end{equation}
where
$$
U_{n,i}=\left\{\overline{u}_n=(u_1,\ldots,u_n) \in|0,i|^n:
\sum_{j=1}^nu_j=i, \,\sum_{j=1}^nju_j=n \right\},
\,B_{\overline{u}_n}=\frac{i!}{\prod\limits_{j=1}^n u_j!}\in\boldsymbol{n}.
$$
In the ring $\ps{k}{x}_s$ of truncated formal power series, for $\sum_{i=1}^sa_ik^i,\sum_{i=1}^sb_ix^i\in\ps{k}{x}_s$ we have
$$
\sum_{i=1}^sa_i\left(\sum_{j=1}^sb_jx^j\right)^i=
\sum_{n=1}^sd_nx^n \qquad \mod x^{s+1}
$$
if and only if $d_n$ for $n\in|1,s|$ are given by~\eqref{7}. As examples of~\eqref{7} we have (cf.~\cite{JablReich2005})
\begin{equation}\label{8}
d_1=a_1b_1,\quad d_2=a_1b_2+a_2b_1^2,\quad d_3=a_1b_3+2a_2b_1b_2+a_3b_1^3.
\end{equation}

To prove our result we need some particular and detailed properties of the operation of substitution described by the formula~\eqref{7}. We collect them here for the convenience of the Reader and we mention that they are  consequences of the equalities, which describe the sets~$U_{n,i}$ in~\eqref{7}.

\begin{lemma}\label{l3}~

\begin{enumerate}[(i)]
\item For $n\geq2$ we have $U_{n,1}=\{(0,\ldots,0,1)\}$ and $U_{n,n}=\{(n,0,\ldots,0)\}$. Then $B_{(0,\ldots,0,1)}=B_{(n,0,\ldots,0)}=1$.
\item If\/ $n\geq3$, $i\in|2,n|$ and $\overline{u}_n\in U_{n,i}$ then $u_j=0$ for each $j\in|n-i+2,n|$.
\item If\/ $n\geq3$ then $U_{n,n-1}=\{(n-2,1,0,\ldots,0)\}$ and  $B_{(n-2,1,0,\ldots,0)}=n-1$.
\item For $n\geq4$ we have $U_{n,n-2}=\{(n-4,2,0,\ldots,0),(n-3,0,1,0,\ldots,0)\}$. Then $B_{(n-4,2,0,\ldots,0)}=\frac{(n-3)(n-2)}{2}$ and $B_{(n-3,0,1,0,\ldots,0)}=n-2$.
\item If\/ $n\geq4$, $i\in|2,n-1|$ and $\overline{u}_n=(u_1,\ldots,u_n)\in U_{n,i}$, then $i-u_1\geq1$.
\end{enumerate}
\end{lemma}

Within studying solutions of Acz{\'e}l-Jabotinsky differential equation we frequently use the particular form of~\eqref{7}. It is a consequence of  Lemma~\ref{l3}~{\em (i)} and~{\em (ii)}, namely we have
\begin{equation}\label{dn}
d_n=a_1b_n +\sum_{i=2}^{n-1} a_i \sum_{\overline{u}_n \in
U_{n,i}} B_{\overline{u}_n} b_1^{u_1}
\prod_{j=2}^{n-i+1}b_j^{u_j}+a_nb_1^n\qquad\mbox{ for }\,n\geq2.
\end{equation}

For fixed integer $l\geq1$, let ${\mathbb{N}}_m$ denote the set of all positive integers $j$ such that $j\equiv1\modd l$. Clearly $\mathbb{N}_1=\mathbb{N}$. For $n\geq l+2$ we put (see Lemma~\ref{l3}~{\em (ii)})
$$
\begin{array}{rcl}
{\widehat{U}}_{n,i}^{l+1}&=&\left\{\overline{u}_n\in U_{n,i}:u_j=0\;\;\mbox{ for }\,j\in|2,l|\cup|n-i+2,n|\,\right\},\\[1ex]
U_{n,i}^m&=&\{\overline{u}_n\in \widehat{U}^{l+1}_{n,i}:u_j=0\;\;
\mbox{ for }\,j\in|l+2,n-i+1|\setminus\mathbb{N}_l\,\}\\[1ex]
&=&\left\{\overline{u}_n\in U_{n,i}:u_j=0\;\;
\mbox{ for }\,j\in(|2,n-i+1|\setminus\mathbb{N}_l)\cup|n-i+2,n|\,\right\}.
\end{array}
$$
Note that $\widehat{U}^2_{n,i}=U^2_{n,i}=U_{n,i}$. We prove now several further properties of~\eqref{dn}. Because of their generality, we will always assume that $l\geq1$ even if for $l=1$ these facts are either obvious or empty satisfied. If $n\geq l+2$ and $i\in|2,n-l|$, we set
$$
\begin{array}{rcl}
\overline{U}^{l+1}_{n,i}&=&\{\overline{u}_n\in \widehat{U}^{l+1}_{n,i}:u_j=0\,\mbox{ for }\,j\in|n-i+1,n|\,\}\\[1ex]
&=&\left\{\overline{u}_n\in U_{n,i}:u_j=0\,\mbox{ for }\,j\in|2,l|\cup|n-i+1,n|\,\right\}.
\end{array}
$$

From equalities describing sets $U_{n,i}$ we derive the following lemma.

\begin{lemma}\label{l3x}
Fix $l\geq1$ and $n\geq l+2$.
\begin{enumerate}[(i)]
\item If\/ $i\in|n-l+1,n-1|$ then $\widehat{U}^{l+1}_{n,i}=\emptyset$, hence also  $U^{l+1}_{n,i}=\emptyset$.
\item If $i\in|2,n-l|$ then
$\widehat{U}^{l+1}_{n,i}=\left\{\left(i-1,0,\ldots,0,
\raise1.2pt\hbox{$\stackrel{n-i+1}{1}$},0,\ldots,0 \right)\right\}\cup\overline{U}^{l+1}_{n,i}$. For $\overline{u}_n=\left(i-1,0,\ldots,0,
\raise1.2pt\hbox{$\stackrel{n-i+1}{1}$},0,\ldots,0 \right)\in \widehat{U}^{l+1}_{n,i}$ we have $B_{\overline{u}_n}=i$.
\end{enumerate}
\end{lemma}

As an elementary consequence of~\eqref{dn} and Lemma~\ref{l3x} we get

\begin{corollary}\label{c1}
If\/ $l\geq1$, $n\geq l+2$ and $b_j=0$ for $j\in|2,l|$, then
$$
\sum_{\overline{u}_n \in
\widehat{U}_{n,i}^{l+1}} B_{\overline{u}_n} b_1^{u_1}
\prod_{j=l+1}^{n-i+1}b_j^{u_j}=ib_1^{i-1}b_{n-i+1}+\sum_{\overline{u}_n \in\overline{U}^{l+1}_{n,i}} B_{\overline{u}_n}b_1^{u_1}\prod_{j=l+1}^{n-i}b_j^{u_j}\,\mbox{ for }\, i\in|2,n-l|,
$$
hence
$$
\begin{array}{rcl}
d_n&=&\displaystyle a_1b_n +\sum_{i=2}^{n-l} a_i \sum_{\overline{u}_n \in
\widehat{U}_{n,i}^{l+1}} B_{\overline{u}_n} b_1^{u_1}
\prod_{j=l+1}^{n-i+1}b_j^{u_j}+a_nb_1^n\\
&=&\displaystyle a_1b_n +\sum_{i=2}^{n-l} a_i \left(ib_1^{i-1}b_{n-i+1}+\sum_{\overline{u}_n\in
\overline{U}^{l+1}_{n,i}}B_{\overline{u}_n}b_1^{u_1} \prod_{j=l+1}^{n-i}b_j^{u_j} \right) +a_nb_1^n.
\end{array}
$$
\end{corollary}

Fix a positive integer $l\geq2$ and put
\begin{equation}\label{xxy}
\begin{array}{ll}
\widetilde{U}_{n,l+1}^{l+1}= \overline{U}^{l+1}_{n,l+1}\;\mbox{ and }\; \widecheck{U}^{l+1}_{n,l+1}=U^{l+1}_{n,l+1}\cap\overline{U}^{l+1}_{n,l+1} & \mbox{ for }\, n\geq2l+2,\\
\widetilde{U}_{n,2l+1}^{l+1}=\widehat{U}_{n,2l+1}^{l+1} & \mbox{ for }\, n\geq 3l+1.
\end{array}
\end{equation}
We prove

\begin{lemma}\label{l5}
If\/ $l\geq2$, $n\in|l,\infty|\setminus\mathbb{N}_l$ and $i\in\{l+1,2l+1\}$ then $\widetilde{U}_{n,i}^{l+1}=\emptyset$.
\end{lemma}
\begin{proof}
Fix $l\geq2$, $n\in|l+2,\infty|\setminus{\mathbb{N}}_l$ and $i\in\{l+1,2l+1\}$. If $\overline{u}_n\in \widetilde{U}_{n,i}^{l+1}$ then for each
$j\in(|2,n-i+1|\setminus\mathbb{N}_l)\cup|n-i+2,n|$ we have $u_j=0$. Hence
$$
n-i=\sum_{j=2}^n(j-1)u_j=\sum_{j\in|2,n-i+1|\cap\mathbb{N}_l}(j-1)u_j,
$$
which implies
$$
n=i+\sum_{j\in|2,n-i+1|\cap\mathbb{N}_l} (j-1)u_j\in{\mathbb{N}}_l.
$$
This contradiction finishes the proof.
\end{proof}

We summarize now properties of substitution in the following lemma.

\begin{corollary}\label{c3}
Fix positive integers $l\geq1$ and $n\geq l+2$. If $b_1=1$ and $b_j=0$ for $j\in|2,n-l|\setminus\mathbb{N}_l$, then
$$
d_n=a_1b_n +\sum_{i=2}^{n-l} a_i \sum_{\overline{u}_n \in
U_{n,i}^{l+1}} B_{\overline{u}_n} b_1^{u_1}
\prod_{j\in|2,n-i+1|\cap\mathbb{N}_l}b_j^{u_j}+a_nb_1^n.
$$
Moreover,
\begin{enumerate}[(i)]
\item for $n\notin\mathbb{N}_l$ and $i\in\{l,2l-1|$ we have
$$
\sum_{\overline{u}_n \in
\widetilde{U}_{n,i}^{l+1}} B_{\overline{u}_n} b_1^{u_1}
\prod_{j\in|l,n-i+1|\cap\mathbb{N}_l}b_j^{u_j}=0,
$$
\item for $n=rl+1\in\mathbb{N}_l$ and
$i=pl+1\in|l+1,n-1|\cap\mathbb{N}_l$ we have
$$
\sum_{\overline{u}_n\in\widetilde{U}^{l+1}_{n,i}}B_{\overline{u}_n} b_1^{u_1}
\prod_{j=l+1}^{n-i+1}b_j^{u_j}=
\sum_{\overline{u}_{rl+1} \in\widecheck{U}_{rl+1,pl+1}^{l+1}}
B_{\overline{u}_{rl+1}} b_1^{u_1}
\prod_{j=1}^{r-p}b_{jl+1}^{u_{jl+1}}.
$$
\end{enumerate}
\end{corollary}

\section{The third Aczel-Jabotynski formal differential equation}\label{AJeqG}

Here and subsequently, if it will not be stated otherwise, $s$ is either a positive integer or $s=\infty$. We prove some properties of the solutions of the third Acz{\'e}l-Jabotinsky differential  equation
$$
(H\circ\Phi)(x)=\frac{\dd\Phi}{\dd x}\cdot H(x)\leqno{({\mathrm{AJ}}_s(H))}
$$
considered here for both, formal power series and $s$-truncated formal power series. In that last case all the equalities are then assumed as to hold $\mbox{mod}\, x^{s+1}$.

\begin{fact}[cf.~\cite{Reich1}]\label{f1}
For a fixed  nonzero $H(x)\in\ps{\boldsymbol{k}}{x}$ the set ${\mathcal{S}}({\mathrm{AJ}}_s(H))$ of all solutions $\Phi$ of~$({\mathrm{AJ}}_s(H))$ is a~group with respect to the substitution.
\end{fact}

\begin{proof}
Indeed, fix $\Phi_1,\Phi_2\in{\mathcal{S}}({\mathrm{AJ}}_s(H))$. Then
$$
\begin{array}{lll}
\lefteqn{
H\left((\Phi_1\circ \Phi_2)(x)\right)
=H\left(\Phi_1(\Phi_2(x))\right)=\left.\frac{\dd\Phi_1}{\dd
x}\right|_{y=\Phi_2(x)}\cdot H(\Phi_2(x))}\null\\[1ex]
&\;\;\;\;\;\;\;\;\;\;\;\;\;\;\;\;\;\;\;\; &
\null\displaystyle =\left.\frac{\dd\Phi_1}{\dd
x}\right|_{y=\Phi_2(x)}\cdot\frac{\dd\Phi_2}{\dd x}\cdot
H(x)=\frac{\dd(\Phi_1\circ\Phi_2)}{\dd x}\cdot H(x).
\end{array}
$$
Hence $\Phi_1\circ\Phi_2\in{\mathcal{S}}({\mathrm{AJ}}_s(H))$. Moreover, if $\Phi\in{\mathcal{S}}({\mathrm{AJ}}_s(H))$ then
$$
H(x)=H((\Phi\circ\Phi^{-1})(x))=H(\Phi(\Phi^{-1}(x)))=\left.\frac{\dd\Phi}{\dd
x}\right|_{y=\Phi^{-1}(x)}\cdot H(\Phi^{-1}(x)),
$$
which is~$({\mathrm{AJ}}_s(H))$ for $\Phi^{-1}$.
\end{proof}

We prove now a connection bewteen ${\mathcal{S}}(\mathrm{AJ}_{\infty}(H_1))$ and ${\mathcal{S}}(\mathrm{AJ}_s(H_2))$ for finite~$s$.

\begin{lemma}\label{lx2}
If $s$ is finite then $\pi^{\infty}_s\left({\mathcal{S}}(\mathrm{AJ}_{\infty}(H))\right)\subset {\mathcal{S}}(\mathrm{AJ}_s(\pi^{\infty}_s(H)))$.
\end{lemma}
\begin{proof}
If $\Phi\in{\mathcal{S}}(\mathrm{AJ}_{\infty}(H))$ then for finite $s$ we have
$$
\pi^{\infty}_s\left((H\circ \Phi)(x)\right)=\pi^{\infty}_s\left(\frac{\dd \Phi}{\dd x}\cdot H(x)\right),
$$
hence
$$
\left((\pi^{\infty}_sH)\circ (\pi^{\infty}_s\Phi)\right)(x)=\pi^{\infty}_s\left(\frac{\dd \Phi}{\dd x}\right)\cdot (\pi^{\infty}_sH)(x)\mod x^{s+1}.
$$
Finally, we have to observe that
$$
\pi^{\infty}_s\left(\frac{\dd \Phi}{\dd x}\right)\cdot (\pi^{\infty}_sH)(x)=
\frac{\dd (\pi^{\infty}_s\Phi)}{\dd x}\cdot (\pi^{\infty}_sH)(x)
\mod x^{s+1},
$$
which means that $\pi^{\infty}_s\Phi\in{\mathcal{S}}(\mathrm{AJ}_s(\pi^{\infty}_sH))$.
\end{proof}

In some cases it is convenient to consider another, equivalent form of the Acz{\'e}l-Jabotinsky equation~$({\mathrm{AJ}}_s(H))$. This is expressed in the following

\begin{fact}[{cf.~\cite[Lemmas~1 and~4]{Reich1}}]\label{f2}
Let $T\in\Gamma_1^s$ be fixed. Then $\Phi\in{\mathcal{S}}({\mathrm{AJ}}_s(H))$ if and only if
$\widetilde{\Phi}=(T^{-1}\circ\Phi\circ
T)\in{\mathcal{S}}({\mathrm{AJ}}_s(\widetilde{H}))$ with $\widetilde{H}(x)=\left(\frac{\dd T}{\dd x}\right)^{-1}(H\circ
T)(x)$.
\end{fact}

\begin{proof}
Let $\Phi$ be a solution of~$({\mathrm{AJ}}_s(H))$ and let
$T\circ\widetilde{\Phi}=\Phi\circ T$. Then
\begin{equation}\label{4}
\left.\frac{\dd T}{\dd u}\right|_{u=\widetilde{\Phi}(x)}
\cdot\frac{\dd\widetilde{\Phi}}{\dd x}=\left.\frac{\dd\Phi}{\dd v}\right|_{v=T(x)}\cdot\frac{\dd T}{\dd x}.
\end{equation}
Moreover, from~$({\mathrm{AJ}}_s(H))$ we have
\begin{equation}\label{5}
(H\circ\Phi\circ T)(x)=\left.\frac{\dd\Phi}{\dd
v}\right|_{v=T(x)}\cdot(H\circ T)(x).
\end{equation}
Consequently, using $\Phi\circ T=T\circ\widetilde{\Phi}$,
from~\eqref{4} and~\eqref{5} we get
$$
\left(\left.\frac{\dd T}{\dd
u}\right|_{u=\widetilde{\Phi}(x)}\right)^{-1}\cdot(H\circ
T\circ\widetilde{\Phi})(x)=\frac{\dd\widetilde{\Phi}}{\dd x}\cdot\left(\frac{\dd T}{\dd x}\right)^{-1}(H\circ T)(x),
$$
which with $\widetilde{H}(x):=\left(\frac{\dd T}{\dd
x}\right)^{-1}(H\circ T)(x)$ gives
\begin{equation}\label{6}
(\widetilde{H}\circ\widetilde{\Phi})(x)=\frac{\dd\widetilde{\Phi}}{\dd x}\cdot\widetilde{H}(x).
\end{equation}
Clearly, the converse mapping $\Phi\mapsto
T\circ\widetilde{\Phi}\circ T^{-1}$ transforms~\eqref{6}
to~$({\mathrm{AJ}}_s(H))$.
\end{proof}

\section{The system of equalities}
\label{sequal}

Calculating both sides of~$({\mathrm{AJ}}_s(H))$ (in the case of finite $s$ all the equalities should be understand as to hold $\mod x^{s+1}$),
we obtain (with $h_{l+1}=1$)
$$
\begin{array}{lll}
\displaystyle\frac{\dd\Phi}{\dd x}H(x)& = & \displaystyle
\left(\sum_{n=1}^snc_nx^{n-1}\right)
\left(x^{l+1}+\sum_{n=l+2}^sh_nx^n\right)\\[3ex]
& = & \displaystyle\sum_{n=l+1}^s
\left(\sum_{i=l+1}^nh_i(n+1-i)c_{n-i+1}\right)x^n,
\end{array}
$$
and, on account of \eqref{7} with $a_j=0$ for $j\in|1,l|$ we get
$$
\begin{array}{lll}
\displaystyle
(H\circ \Phi)(x))& = &\displaystyle \sum_{n=l+1}^sh_n\left(\sum_{i=1}^sc_ix^i \right)^n \\[3ex]
& = & \displaystyle \sum_{n=l+1}^s\left(\sum_{i=l+1}^nh_i\sum_{\overline{u}_n\in
U_{n,i}}B_{\overline{u}_n}\prod_{j=1}^{n-i+1}c_j^{u_j}\right)x^n
\end{array}
$$
with $h_l=1$. Then~$({\mathrm{AJ}}_s(H))$ is equivalent to
\begin{equation}\label{9}
\sum_{i=l+1}^nh_i\sum_{\overline{u}_n\in
U_{n,i}}B_{\overline{u}_n}\prod_{j=1}^{n-i+1}c_j^{u_j}=
\sum_{i=l+1}^nh_i(n-i+1)c_{n-i+1}\,\mbox{ for every }\,n\in|l+1,s|
\end{equation}
with $h_{l+1}=1$. In particular, for $n=l+1$ we obtain $c_1^{l+1}=c_1$. Hence, $c_1\in\boldsymbol{k}^\star$ can be arbitrarily fixed for $l=0$ and $c_1\in\boldsymbol{E}_l$ for each positive integer $l$. Therefore  determining solutions of~$({\mathrm{AJ}}_s(H))$ for both, finite and
infinite~$s$, we consider two distinct cases:
\begin{itemize}
\item[(C1)] $l=0$ i.e. $H(x)=x+\sum_{n=2}^sh_nx^n$,
\item[(C2)] $l\geq1$, i.e. $H(x)=x^{l+1}+\sum_{n=l+2}^s
h_nx^n$.
\end{itemize}
In the first case, from ~\eqref{9} on account of~\eqref{8} and~\eqref{dn} we get
$$
\begin{array}{lll}
\lefteqn{c_2+h_2c_1^2=2c_2+h_2c_1 \qquad\qquad\qquad\qquad\quad\, (\mbox{for }\,n=2 \mbox{ provided }\,s\geq2),}\null\\[1ex]
\lefteqn{c_3+2h_2c_1c_2+h_3c_1^3=3c_3+2h_2c_2+h_3c_1 \quad (\mbox{for }\,n=3 \mbox{ provided }\,s\geq3),}\null\\
\lefteqn{c_n+\sum_{i=2}^{n-1}h_i\sum_{\overline{u}_n\in
U_{n,i}}B_{\overline{u}_n}\prod_{j=1}^{n-i+1}c_j^{u_j}+h_nc_1^n}\null\\
& \qquad\qquad\qquad\qquad & \null\displaystyle =nc_n+\sum_{i=2}^nh_i(n-i+1)c_{n-i+1}
\quad \mbox{ for  }\,n\in|4,s|.
\end{array}
$$
This system can easily be reduced to
\begin{equation}\label{11}
\begin{array}{lll}
\lefteqn{c_2=h_2(c_1^2-c_1) \qquad\mbox{ provided } \,s\geq2,}\null\\[1ex]
\lefteqn{2c_3=h_3(c_1^3-c_1)+2h_2c_2(c_1-1)\qquad \mbox{ provided } \,s\geq3,}\null\\[1ex]
\lefteqn{(n-1)c_n=h_n(c_1^n-c_1)+}\null\\
& \quad & \null\displaystyle
\sum_{i=2}^{n-1}h_i\left(\sum_{\overline{u}_n\in
U_{n,i}}B_{\overline{u}_n}c_1^{u_1}\prod_{j=2}^{n-i+1} c_j^{u_j}-(n-i+1)c_{n-i+1}\right),\;n\in|4,s|,
\end{array}
\end{equation}
and we can find here the form of a solution by some recurrent procedure. In the second case, using Lemma~\ref{l3}~{\em(i)}, {\em(iii)},~{\em(iv)} and~\eqref{dn} we rewrite~\eqref{9} in the form
\begin{equation}\label{15}
\begin{array}{lll}
\lefteqn{(l+1)c_1^lc_2+h_{l+2}c_1^{l+2}=2c_2+h_{l+2}c_1  \qquad\quad(n=l+1 \mbox{ provided }s\geq l+2),}\null\\[1ex]
\lefteqn{\frac{l(l+1)}{2}c_1^{l-1}c_2^2+(l+1)c_1^lc_3+(l+2)c_1^{l+1}c_2+ h_{l+3}c_1^{l+3}}\null\\[1ex]
&  & \null\displaystyle\quad=
3c_2+2h_{l+2}c_2+h_{l+3}c_1  \qquad\qquad\qquad  (n=l+3 \mbox{ provided }s\geq l+3),\\[1ex]
\lefteqn{\sum_{\overline{u}_n\in
U_{n,l+1}}B_{\overline{u}_n}\prod_{j=1}^{n-l}c_j^{u_j}
+\sum_{i=l+2}^{n-1}h_i\sum_{\overline{u}_n\in
U_{n,i}}B_{\overline{u}_n}\prod_{j=1}^{n-i+1}c_j^{u_j}+
h_nc_1^n}\null\\[1ex]
&  & \null\displaystyle\quad =(n-l)c_{n-l}+\sum_{i=l+2}^{n-1}(n-i+1)h_ic_{n-i+1} +h_nc_1,\qquad\quad n\in|l+4,s|.
\end{array}
\end{equation}
Note that $U_{n,l+1}=\widehat{U}^2_{n,l+1}$. By Corollary~\ref{c1} with $l=1$ and $i=l+1$ we get
$$
\begin{array}{lll}
\lefteqn{\sum_{\overline{u}_n\in
U_{n,l+1}}B_{\overline{u}_n}c_1^{u_1}\prod_{j=2}^{n-l}c_j^{u_j}}\null\\
& & \null\displaystyle =
\sum_{\overline{u}_n\in
\widehat{U}^2_{n,l+1}}B_{\overline{u}_n}c_1^{u_1} \prod_{j=2}^{n-l}c_j^{u_j}=
(l+1)c_1^lc_{n-l}+\sum_{\overline{u}_n\in
\overline{U}^2_{n,l+1}}B_{\overline{u}_n}\prod_{j=1}^{n-l-1}c_j^{u_j},
\end{array}
$$
thus from~\eqref{15} we obtain
$$
\begin{array}{lll}
\lefteqn{2c_2-(l+1)c_1^lc_2=h_{l+2}(c_1^{l+2}-c_1),}\null\\
\lefteqn{3c_3-(l+1)c_1^lc_3=\frac{l(l+1)}{2}c_1^{l-1}c_2^2+h_{l+3} (c_1^{l+3}-c_1)+h_{l+2}((l+2)c_1^{l+1}c_2-2c_2),}\\
\lefteqn{(n-l)c_{n-l}-(l+1)c_1^lc_{n-l}=
\sum_{\overline{u}_n\in
\overline{U}^2_{n,l+1}}B_{\overline{u}_n}c_1^{u_1} \prod_{j=2}^{n-l}c_j^{u_j} +h_n(c_1^n-c_1)}\null\\
& & \null\displaystyle \quad+
\sum_{i=l+2}^{n-1}h_i\left(\sum_{\overline{u}_n\in
U_{n,i}}B_{\overline{u}_n}c_1^{u_1}\prod_{j=2}^{n-i+1} c_j^{u_j}-(n-i+1)c_{n-i+1} \right), \,n\in|l+4,s|.
\end{array}
$$
Since $c_1^l=1$, we get
\begin{equation}\label{16}
\begin{array}{lll}
\lefteqn{(1-l)c_2=h_{l+2}(c_1^2-c_1)\qquad\qquad\qquad(n=l+2 \mbox{ provided }s\geq l+2)}\null\\[1ex]
\lefteqn{(2-l)c_3=\frac{l(l+1)}{2}c_1^{l-1}c_2^2+}\null\\
& & \null\displaystyle \quad h_{l+2}(c_1^3-c_1)
+h_{l+2}c_2((l+2)c_1-2) \quad (n=l+3 \mbox{ provided }s\geq l+3),\\
\lefteqn{(n-1-2l)c_{n-l}= \sum_{\overline{u}_n\in
\overline{U}^2_{n,l+1}}B_{\overline{u}_n}c_1^{u_1}\prod_{j=2}^{n-l} c_j^{u_j}+h_n(c_1^n-c_1)+}\null\\
& & \null\displaystyle
\quad \sum_{i=l+2}^{n-1}h_i\left(\sum_{\overline{u}_n\in U_{n,i}}B_{\overline{u}_n}c_1^{u_1}\prod_{j=2}^{n-i+1}c_j^{u_j}-
(n-i+1)c_{n-i+1}\right),\, n\in|l+4,s|.
\end{array}
\end{equation}
Observe that for either $l=1$ or $s\leq 2l$ we also can get the form of a solution by a recurrent procedure. The case $l\geq2$ and $s\geq 2l+1$ is much more complicated. Indeed, for $n=2l+1$ we are able to determine by recurrent method the form of $c_j$ for $j\in\{2,\ldots,l\}$ and then we have $(n-1-2l)c_{n-l}=0c_{l+1}=0$, but we are not able to verify whether the right hand side of this equality is also equal to zero. Therefore we transform the Acz{\'e}l-Jabotinsky differential equation to a~simpler form, where  that problem will not appear.

\section{Normal forms of the third Acz{\'e}-Jabotinsky formal differential equation}

Fix $s\in\mathbb{N}\cup\{\infty\}$, $l\in\mathbb{N}\cup\{0\}$ and $H(x)=x^{l+1}+\sum_{i=l+2}^sh_ix^i\in\ps{k}{x}_s$. We prove that the Acz{\'e}l-Jabotinsky formal differential equation~$({\mathrm{AJ}}_s(H))$ can be transformed (see Fact~\ref{f2}) to possibly simple normal form. Put $\widetilde{H}(x)=x^{l+1}+\delta x^{2l+1}$. We find $T(x)=x+\sum_{i=2}^sv_ix^i$ which transform~$({\mathrm{AJ}}_s(H))$ to
$$
(\widetilde{\Phi}(x))^{l+1}+\delta(\widetilde{\Phi}(x))^{2l+1}=\frac{\dd \widetilde{\Phi}}{\dd x}\cdot\left(x^{l+1}+\delta x^{2l+1}\right),\leqno{({\mathrm{AJ}}_s(\widetilde{H}))}
$$
where for a finite $s$ the above equality should be understood as to hold $\modd x^{s+1}$ with an additional assumption $\delta=0$ for either $l=0$ or $s<2l+1$.

\begin{lemma}\label{l6}
For each $H(x)=x^{l+1}+\sum_{i=l+2}^sh_ix^i\in\ps{k}{x}_s$ there exists $T\in\Gamma_1^s$ which transforms~$({\mathrm{AJ}}_s(H))$ to~$({\mathrm{AJ}}_s(\widetilde{H}))$. In the cases either $l=0$ or $s\leq2l$ it can be done for $\delta=0$ and then the transformation $T$ is unique.
\end{lemma}

\begin{proof}
It is enough to find (cf.~Fact~\ref{f2}) $T(x)=x+\sum_{i=2}^sv_ix^i\in\Gamma_1^s$ such that $(H\circ T)(x)=\frac{\dd T}{\dd x}(x^l+\delta x^{2l-1})$. Thus we have (with $h_l=1$)
$$
\sum_{i=l+1}^sh_i\left(X+\sum_{j=2}^s v_jx^j\right)^i=\left(1+
\sum_{n=2}^skv_nx^{n-1}\right)(x^{l+1}+\delta x^{2l+1}),
$$
which leads to the following system of equalities
\begin{equation}\label{16a}
\begin{array}{lll}
\lefteqn{\displaystyle \sum_{i=l+1}^{n-1}h_i\sum_{\overline{u}_n\in
U_{n,i}}B_{\overline{u}_n}\prod_{j=2}^{n-i+1}v_j^{u_j} +h_n=(n-l)v_{n-l},\quad
n\in|l+2,2l|,}\null\\[1ex]
\lefteqn{\displaystyle \sum_{i=l+1}^{2l}h_i\sum_{\overline{u}_{2l+1}\in
U_{2l+1,i}}B_{\overline{u}_{2l+1}}\prod_{j=2}^{2l+2-i}v_j^{u_j}+h_{2l+1}}\null\\  & & \null\displaystyle =(l+1)v_{l+1}+\delta, \qquad\qquad   \mbox{ provided }\,s\geq2l+1,\null\\
\lefteqn{\displaystyle \sum_{i=l+1}^{n-1}h_i\sum_{\overline{u}_n\in
U_{n,i}}B_{\overline{u}_n}\prod_{j=2}^{n-i+1}v_j^{u_j}+h_n}\null\\[1ex]
& \qquad\qquad\qquad \ & \null\displaystyle =(n-l)v_{n-l}+ \delta(n-2l)v_{n-2l},\quad n\in|2l+2,s|.
\end{array}
\end{equation}
We consider firstly the case $l=0$. Clearly $|l+2,2l+1|=|2,1|=\emptyset$ and we can put $\delta=0$. On account of~\eqref{dn}, from~\eqref{16a} we obtain
$$
v_n+\sum_{i=2}^{n-1}h_i\sum_{\overline{u}_n\in U_{n,i}}B_{\overline{u}_n}\prod_{j=2}^{n-i+1}v_j^{u_j}+h_n=nv_n,\quad n\in|2,s|,
$$
which implies the folloving recurrent formula
\begin{equation}\label{16b}
\left\{\begin{array}{l}
v_2=h_2\\
\displaystyle v_n=\frac{1}{n-1}\left(\sum_{i=2}^{n-1}h_i\sum_{\overline{u}_n\in U_{n,i}}B_{\overline{u}_n}\prod_{j=2}^{n-i+1}v_j^{u_j}+h_n\right) \quad\mbox{ for }\,n\in|3,s|.
\end{array}\right.
\end{equation}
Assume now $l\geq1$. On account of Corollary~\ref{c1}~{\em(i)} with $l=1$ and $i=l+2$, from~\eqref{16a} we get ($U_{n,l+1}=\widehat{U}^2_{n,l+1}$)
$$
\begin{array}{lll} \lefteqn{(l+1)v_{n-l}+\sum_{\overline{u}_n\in
\overline{U}^2_{n,l+1}}B_{\overline{u}_n}\prod_{j=2}^{n-l-1}v_j^{u_j} +\sum_{i=l+2}^{n-1}h_i\sum_{\overline{u}_n\in
U_{n,i}}B_{\overline{u}_n}\prod_{j=2}^{n-i+1}v_j^{u_j}}\null\\[2ex]
& & \null\displaystyle
+h_n=(n-l)v_{n-l},\qquad\qquad\qquad\quad   n\in|l+2,2l|,\\[1ex]
\lefteqn{(l+1)v_{l+1}+\sum_{\overline{u}_{2l+1}\in
\overline{U}^2_{2l+1,l+1}}B_{\overline{u}_n}\prod_{j=2}^{l} v_j^{u_j}+\sum_{i=l+2}^{2l}h_i\sum_{\overline{u}_{2l+1}\in
U_{2l+1,i}}B_{\overline{u}_n}\prod_{j=2}^{2l+2-i} v_j^{u_j}}\null\\[2ex]
& & \null\displaystyle +h_{2l+1} =(l+1)v_{l+1}+ \delta, \qquad\qquad   \mbox{ provided }\,s\geq2l+1\\[1ex]
\lefteqn{(l+1)v_{n-l}+\sum_{\overline{u}_n\in
\overline{U}^2_{n,l+1}}B_{\overline{u}_n}\prod_{j=2}^{n-l-1}v_j^{u_j} +\sum_{i=l+2}^{n-1}h_i\sum_{\overline{u}_n\in
U_{n,i}}B_{\overline{u}_n}\prod_{j=2}^{n-i+1}v_j^{u_j}}\null\\[2ex]
& \qquad\qquad\qquad & \null\displaystyle +h_n=(n-l+1)v_{n-l+1}+ \delta(n-2l)v_{n-2l}, \quad n\in|2l+2,s|,
\end{array}
$$
or, equivalently
\begin{equation}\label{17}
\begin{array}{lll}
\lefteqn{(n-2l-1)v_{n-l}=\sum_{\overline{u}_n\in
\overline{U}^2_{n,l+1}}B_{\overline{u}_n}\prod_{j=2}^{n-l-1}v_j^{u_j}}\null\\
& & \null\displaystyle+
\sum_{i=l+2}^{n-1}h_i\sum_{\overline{u}_n\in
U_{n,i}}B_{\overline{u}_n}\prod_{j=2}^{n-i+1}v_j^{u_j}+h_n,\qquad\qquad n\in|l+2,2l|,\\[1ex]
\lefteqn{\displaystyle
\delta=\sum_{\overline{u}_{2l+1}\in
\overline{U}^2_{2l+1,l+1}}B_{\overline{u}_n}\prod_{j=2}^{l}v_j^{u_j}}\null\\
& & \null\displaystyle +
\sum_{i=l+2}^{2l}h_i\sum_{\overline{u}_{2l+1}\in
U_{2l+1,i}}B_{\overline{u}_n}\prod_{j=2}^{2l+2-i}v_j^{u_j}+h_{2l+1}, \quad \mbox{ provided }\,s\geq2l+1,\\[1ex]
\lefteqn{(n-2l-1)v_{n-l}=\sum_{\overline{u}_n\in
\overline{U}^2_{n,l+1}}B_{\overline{u}_n}\prod_{j=2}^{n-l-1}v_j^{u_j}}\null\\
& & \null\displaystyle +\sum_{i=l+2}^{n-1}h_i\sum_{\overline{u}_n\in
U_{n,i}}B_{\overline{u}_n}\prod_{j=2}^{n-i+1}v_j^{u_j}+h_n-
\delta(n-2l)v_{n-2l},\; n\in|2l+2,s|.
\end{array}
\end{equation}
Lut us consider the case $l=1$. We have $|l+2,2l|=|3,2|=\emptyset$ and $\widetilde{U}_{3,2}=\emptyset$. Moreover, for $n=4$ we have
$$
\sum_{\overline{u}_4\in
\overline{U}^2_{4,2}}B_{\overline{u}_4}\prod_{j=2}^{4-2}v_j^{u_j}+\sum_{i=3}^3h_i \sum_{\overline{u}_4\in
U_{4,i}}B_{\overline{u}_4}\prod_{j=2}^{4-i+1}v_j^{u_j}+h_4=v_2^2+3h_3v_2+h_4,
$$
hence from~\eqref{17} we obtain
$$
\begin{array}{lll}
\lefteqn{\delta=h_3, \quad\mbox{ provided }\,s\geq3,}\null\\[1ex]
\lefteqn{v_3=v_2^2+3h_3v_2+h_4-2\delta v_{2},\quad\mbox{ provided }\,s\geq4, }\null\\
\lefteqn{(n-3)v_{n-1}=\sum_{\overline{u}_n\in
\overline{U}^2_{n,2}}B_{\overline{u}_n}\prod_{j=2}^{n-2}v_j^{u_j}}\null\\
& \qquad\qquad\quad & \null\displaystyle +\sum_{i=3}^{n-1}h_i\sum_{\overline{u}_n\in
U_{n,i}}B_{\overline{u}_n}\prod_{j=2}^{n-i+1}v_j^{u_j}+h_n-
(n-2)\delta v_{n-2},\; n\in|5,s|.
\end{array}
$$
Thus $\delta=h_3$, $v_2$ can be arbitrarily fixed. Moreover, we obtain the following recurrence formula
$$
\left\{\begin{array}{lll}
\lefteqn{v_3=v_2^2+h_3v_2+h_4,}\null\\
\lefteqn{v_{n-1}=\frac{1}{n-3}\left(\sum_{\overline{u}_n\in
\overline{U}^2_{n,2}}B_{\overline{u}_n}\prod_{j=2}^{n-2}v_j^{u_j}\right.}\null\\
& \qquad\; & \displaystyle\null \left.+ \sum_{i=3}^{n-1}h_i\sum_{\overline{u}_n\in
U_{n,i}}B_{\overline{u}_n}\prod_{j=2}^{n-i+1}v_j^{u_j}+h_n-
(n-2)h_3 v_{n-2}\right)\;\mbox{ for }\, n\in|5,s|,
\end{array}
\right.
$$
and we fix $v_n\in\boldsymbol{k}$ arbitrarily.

Finally, assume that $l\geq2$. We determine $v_k$ for $k\in|2,l|$ recursively. Indeed, $\widetilde{U}_{l+2,l+1}=\emptyset$, so from~\eqref{17} for $n=l+2$ we get
$$
\left\{\begin{array}{lll}
\lefteqn{v_2=-\frac{1}{l-1}h_{l+2},}\null\\
\lefteqn{v_{n-l}=-\frac{1}{2l+1-n}\left( \sum_{\overline{u}_n\in
\overline{U}^2_{n,l+1}}B_{\overline{u}_n}\prod_{j=2}^{n-l-1}v_j^{u_j}\right.}\null\\
& \qquad\qquad & \null\displaystyle \left.+
\sum_{i=l+2}^{n-1}h_i\sum_{\overline{u}_n\in
U_{n,i}}B_{\overline{u}_n}\prod_{j=2}^{n-i+1}v_j^{u_j}+h_n \right)\quad\mbox{ for }\,n\in|l+3,2l|.
\end{array}\right.
$$
Next, put (see~\eqref{17} for $n=2l+1$)
$$
\delta=\sum_{\overline{u}_{2l+1}\in
\overline{U}^2_{2l+1,l+1}}B_{\overline{u}_n}\prod_{j=2}^{l+1}v_j^{u_j}+
\sum_{k=l+2}^{2l+1}h_k\sum_{\overline{u}_{2l+1}\in
U_{2l+1,k}}B_{\overline{u}_n}\prod_{j=2}^{2l+2-k}v_j^{u_j}
$$
provided $s\geq2l+1$ and fix $v_l\in\boldsymbol{k}$ arbitrarily. Finally, from~\eqref{17} we find
$$
\begin{array}{lll}
\lefteqn{
v_{n-l}=\frac{1}{n-2l-1}\left(\sum_{\overline{u}_n\in\overline{U}^2_{n,l+1}} B_{\overline{u}_n}\prod_{j=2}^{n-l-1}v_j^{u_j} +\sum_{i=l+2}^{n-1}h_i\sum_{\overline{u}_n\in
U_{n,i}}B_{\overline{u}_n}\prod_{j=2}^{n-i+1}v_j^{u_j}\right.}\null\\[1ex]
& \qquad\qquad\qquad\qquad\qquad\qquad & \null\displaystyle \left. +\rule{0pt}{7mm}h_n-
\delta(n-2l)v_{n-2l}\right)\quad\mbox{ for }\, n\in|2l+2,s|,
\end{array}
$$
and we fix $v_n\in\boldsymbol{k}$ arbitrarily for $n\in|s-l+1,s|$ whenever $s<\infty$.
\end{proof}

\section{Groups with simple algebraic structure}

We give here technical results connected with direct and semi-direct products of some groups. We begin with

\begin{fact}\label{fxy1}
Let $(G,\cdot)$ be a group and let $H_1,H_2\subset G$ be subgroups such that $G=H_1\cdot H_2$ and $H_1\cap H_2=\{1\}$.
\begin{enumerate}[(i)]
\item If $H_2$ is a normal divisor in~$G$, then $G=H_1\overline{\odot} H_2$ is a semi-direct product of the subgroups $H_1$ and $H_2$.
\item If $H_1,H_2$ are normal in~$G$, then $G=H_1\odot H_2$ is a direct product of $H_1$ and~$H_2$.
\item If $h_1h_2=h_2h_1$ for all $h_1\in H_1$ and $h_2\in H_2$, then both $H_1$ and $H_2$ are normal divisors in~$G$.
\end{enumerate}
\end{fact}

\begin{lemma}\label{lxy1}
Let $(G,\cdot)$ be a group and let $H\subset G$ be a subgroup. Assume that $\varphi:G\to H$ is a homomorphism such that $\varphi|_H=\id_H$. Then $G=H\overline{\odot}\ker\varphi$. Moreover, if $h_1h_2=h_2h_1$ for all $h_1\in H$ and $h_2\in\ker\varphi$, then $G=H\odot\ker\varphi$.
\end{lemma}

\begin{proof}
For every $g\in G$ we have $\varphi(g)^{-1}\in H$, hence $\varphi(\varphi(g)^{-1})=\varphi(g)^{-1}$ and $\varphi\left(\varphi(g)^{-1}\cdot g\right)=\varphi\left(\varphi(g)^{-1}\right)\cdot\varphi(g)= \varphi(g)^{-1}\cdot\varphi(g)=1$, thus
$$
g=\varphi(g)\cdot\left(\varphi(g)^{-1}\cdot g\right)\in H\cdot\ker\varphi,
$$
so $G=H\cdot\ker\varphi$. Moreover, $\ker\varphi$ is a normal divisor in~$G$ and $\varphi(h)=1$ for $h\in H$ implies $h=1$, so $H\cap\ker\varphi=\{1\}$. On account of Fact~\ref{fxy1}~{\em (i)} we get $G=H\overline{\odot}\ker\varphi$. Finally, the last statement if a consequence of Fact~\ref{fxy1}~{\em (iii)} and~{\em (ii)}.
\end{proof}

We consider here some groups with a simple algebraic structure, which appear in an algebraic characterizations of groups of solutions of the third Acz{\'e}l-Jabotinsky formal differential equation. In these characterizations we use additive group $(\boldsymbol{k},+)$ as well as multiplicative groups $(\boldsymbol{k}^*,\cdot)$ and $(\boldsymbol{E}_l,\cdot)$ with $l\in\mathbb{N}$. In a~product $\boldsymbol{E}_l\times\boldsymbol{k}$ we introduce an operation $\diamond:(\boldsymbol{E}_l\times\boldsymbol{k})\times(\boldsymbol{E}_l \times\boldsymbol{k})\to\boldsymbol{E}_l\times\boldsymbol{k}$,
$$
(c_1,c_2)\diamond(d_1,d_2)=(c_1d_1,c_1d_2+d_1c_2)\qquad \mbox{ for }\,(c_1,c_2),(d_1,d_2)\in\boldsymbol{E}_l\times\boldsymbol{k}.
$$
Obviously $(\boldsymbol{E}_l\times\boldsymbol{k},\diamond)$ is a commutative group which is, in fact, a direct product $(\boldsymbol{E}_l\times\{0\})\diamond(\{1\}\times\boldsymbol{k})$ of the subgroups $\boldsymbol{E}_l\times\{0\}$ and $\{1\}\times\boldsymbol{k}$.

For fixed $l\in\mathbb{N}$ we put
$$
\begin{array}{l}
\displaystyle \Gamma^{l,l}=\{c_1x+c_{l+1}x^{l+1}\in\Gamma^{l+1}: c_1\in\boldsymbol{E}_l, c_l\in\boldsymbol{k}\,\},\\ [.5ex]
\displaystyle \Gamma^{l,l}_1=\{x+c_{l+1}x^{l+1}\in\Gamma^{l+1,l+1}: c_{l+1}\in\boldsymbol{k}\,\}.
\end{array}
$$

\begin{lemma}\label{lx3}
For $l\in\mathbb{N}$ we have the following properties:
\begin{enumerate}[(i)]
\item $\Gamma^{l,l}\subset\Gamma^{l+1}$ is a commutative subgroup which is isomorphic to $(\boldsymbol{E}_l\times\boldsymbol{k},\diamond)$,
\item $\Gamma_1^{l,l}$ is a subgroup in $\Gamma^{l,l}$ which is isomorphi to $(\boldsymbol{k},+)$,
\item $\Gamma^{l,l}={\mathcal{L}}_l\circledcirc\Gamma^{l,l}_1$ is a~direct product of the subgroups ${\mathcal{L}}_l$ and $\Gamma^{l,l}_1$.
\end{enumerate}
\end{lemma}

\begin{proof}
Let $f(x)=c_1x+c_{l+1}x^{l+1},g(x)=d_1x+d_{l+1}x^{l+1}\in \Gamma^{l,l}$. Then $d_1^{l+1}=d_1$ and
$$
(f\circ g)(x)=(c_1d_1)x+(c_1d_{l+1}+d_1c_{l+1})x^{l+1}\mod x^{l+2},
$$
so $f\circ g\in\Gamma^{l,l}$ and $f^{\circ-1}(x)=c_1^{-1}x+c_1^{-1}c_{l+1}x^{l+1}\in\Gamma^{l,l}$. Hence $\Lambda:\boldsymbol{E}_l\times\boldsymbol{k}\to\Gamma^{l,l}$, $\Lambda(c_1,c_{l+1})=c_1x+c_{l+1}x^{l+1}$ for $(c_1,c_{l+1})\in \boldsymbol{E}_l\times\boldsymbol{k}$ is a groups isomorphism. Furthermore, $\Gamma_1^{l,l}=\ker\pi^{l+1}_1|_{\Gamma^{l,l}}$ and clearly $\Gamma^{l,l}_1\cong\boldsymbol{k}$. Finally, every $f\in\Gamma^{l,l}$ with $\pi^{l+1}_1f=L_c\in{\mathcal{L}}_l$ we can uniquely write in the form $f=L_c\circ\left(L_{c^{-1}}\circ f\right)$ and $L_{c^{-1}}\circ f\in\Gamma^{l,l}_1$, which finishes the proof.
\end{proof}

In a description of solutions of the Acz{\'e}l-Jabotinsky differential equation we use also the following groups. For $l,s\in\mathbb{N}$ with and $2l+1\leq s$ let us consider the product $\boldsymbol{E}_l\times\boldsymbol{k}^{l+1}$ with an operation
$\overline{\diamond}:(\boldsymbol{E}_l\times\boldsymbol{k}^{l+1}) \times(\boldsymbol{E}_l\times\boldsymbol{k}^{l+1})\to\boldsymbol{E}_{l+1} \times\boldsymbol{k}^{l+1}$,
$$
(c_1,(c_j)_{j\in\{l\}\cup|s-l+1,s|})\overline{\diamond}(d_1,(d_j)_{j\in\{l\} \cup|s-l+1,s|}) =(c_1d_1,(c_1d_j+d_1^jc_j)_{j\in\{l\}\cup|s-l+1,s|})
$$
for $(c_1,(c_j)_{j\in\{l\}\cup|s-l+1,s|}),(d_1,(d_j)_{j\in\{l\}\cup|s-l+1,s|}) \in\boldsymbol{E}_l\times\boldsymbol{k}^{l+1}$.

Similarly, if $l,s\in\mathbb{N}$ with $l+1\leq s\leq2l$, we use the product $\boldsymbol{E}_l\times\boldsymbol{k}^l$ with an operation
$\widehat{\diamond}:(\boldsymbol{E}_l\times\boldsymbol{k}^l) \times(\boldsymbol{E}_l\times\boldsymbol{k}^l)\to \boldsymbol{E}_l\times\boldsymbol{k}^l$,
$$
(c_1,(c_j)_{j\in|s-l+1,s|})\widehat{\diamond}(d_1,(d_j)_{j\in|s-l+1,s|}) =(c_1d_1,(c_1d_j+d_1^jc_j)_{j\in|s-l+1,s|})
$$
for $(c_1,(c_j)_{j\in|s-l+1,s|}),(d_1,(d_j)_{j\in|s-l+1,s|}) \in\boldsymbol{E}_l\times\boldsymbol{k}^l$. Observe that for $l\geq2$ the groups $(\boldsymbol{E}_l\times\boldsymbol{k}^{l+1},\overline{\diamond})$ and $(\boldsymbol{E}_l\times\boldsymbol{k}^l,\overline{\diamond})$ need not be commutative. This depends on the cardinality of the group $\boldsymbol{E}_l$. More precisely, if $\{1\}\subsetneq\boldsymbol{E}_l$ then $d_1^j\neq d_1$ for any $1\neq d_1\in\boldsymbol{E}_l$ and some $j\in|s-l+1,s|$.

\section{The algebraic structure of groups of solutions}
\label{algstruct}

We describe now algebraic properties of groups of solutions of~$({\mathrm{AJ}}_s(H))$ for $s\in\mathbb{N}\cup\{\infty\}$. An important tool is here the inner automorphism ${\mathcal{A}}_T:\Gamma^s\to\Gamma^s$, ${\mathcal{A}}_T(F)=T^{-1}\circ F\circ T$ for $F\in\Gamma^s$, where $T\in\Gamma^s_1$ is fixed. We begin with an obvious consequence of Fact~\ref{f2} and Lemma~\ref{l6}.

\begin{corollary}\label{cstrAJ}
Fix $s\in\mathbb{N}\cup\{\infty\}$, $l\in|0,s-1|$, $H(x)=x^{l+1}+\sum_{i=l+2}^sh_ix^i\in\ps{k}{x}_s$. Let $T\in\Gamma^s_1$ transforms~$({\mathrm{AJ}}_s(H))$ to~$({\mathrm{AJ}}_s(\widetilde{H}))$, where $\widetilde{H}(x)=x^{l+1}+\delta x^{2l+1}$ with suitable $\delta\in\boldsymbol{k}$. Then ${\mathcal{A}}_T|_{{\mathcal{S}}({\mathrm{AJ}}_s(H))}: {\mathcal{S}}({\mathrm{AJ}}_s(H))\to {\mathcal{S}}({\mathrm{AJ}}_s(\widetilde{H}))$ is a groups isomorphism.
\end{corollary}

Let us fix firstly $H(x)=x+\sum_{i=2}^sh_ix^i\in\ps{k}{x}_s$. We prove that the group ${\mathcal{S}}({\mathrm{AJ}}_s(H))$ of solutions of~$({\mathrm{AJ}}_s(H))$ is isomorphic then to $(\boldsymbol{k}^\star,\cdot)$. From Corollary~\ref{cstrAJ} and Lemma~\ref{l6} we get

\begin{theorem}\label{tNAJ1}
Fix $s\in\mathbb{N}\cup\{\infty\}$, $H\in\Gamma^s_1$ and let $T\in\Gamma^s_1$ transforms~$({\mathrm{AJ}}_s(H))$ to~$({\mathrm{AJ}}_s(\widetilde{H}))$, where $\widetilde{H}=L_1$. Then
\begin{enumerate}[(i)]
\item ${\mathcal{A}}_T({\mathcal{S}}({\mathrm{AJ}}_s(H)))= {\mathcal{S}}({\mathrm{AJ}}_s(\widetilde{H}))= {\mathcal{S}}({\mathrm{AJ}}_s(L_1))={\mathcal{L}}\cong (\boldsymbol{k}^\star,\cdot)$,
\item ${\mathcal{S}}({\mathrm{AJ}}_s(H))= {\mathcal{A}}_{T^{-1}}({\mathcal{L}})=\left\{ T\circ L_c\circ T^{-1}:c\in\boldsymbol{k}^\star\right\}\cong(\boldsymbol{k}^\star,\cdot)$.
\end{enumerate}
\end{theorem}

\begin{proof}
On account of Lemma~\ref{l6} we know that there exists $T\in\Gamma^s_1$ which transforms~$({\mathrm{AJ}}_s(H))$ to~$({\mathrm{AJ}}_s(\widetilde{H}))$ with $\widetilde{H}=L_1$, i.e. to
$$
\Phi(x)=\frac{\dd \Phi}{\dd x}\cdot x.
$$
As one can easily to check, $\Phi(x)=cx$ for every $c\in\boldsymbol{k}^\star$ are the only solutions of this formal differential equations, that is ${\mathcal{S}}({\mathrm{AJ}}_s(\widetilde{H}))={\mathcal{L}}$. Obviously ${\mathcal{L}}\cong \boldsymbol{k}^\star$. The second statement is a consequence of the first one.
\end{proof}

Now we proceed with the second case, that is we study the structure of the group of solutions of ${\mathrm{AJ}}_s(H)$ for a fixed $H(x)=x^{l+1}+\sum_{i=l+2}^sh_ix^i$, where $s\in\mathbb{N}\cup\{\infty\}$, $l\in|1,s-1|$. Also here we can transform~$({\mathrm{AJ}}_s(H))$ to its normal form~$({\mathrm{AJ}}_s(\widetilde{H}))$ with $\widetilde{H}(x)=x^{l+1}+\delta x^{2l+1}$, where $\delta\in\boldsymbol{k}$ and $\delta=0$ provided $s\leq2l$. Let us note, that ${\mathcal{L}}_l\subset{\mathcal{S}}({\mathrm{AJ}}_s(\widetilde{H}))$ is a subgroup and $\pi^s_1\Phi\in{\mathcal{L}}_l$ for every $\Phi\in{\mathcal{S}}({\mathrm{AJ}}_s(\widetilde{H}))$.
Put
$$
{\mathcal{S}}_1({\mathrm{AJ}}_s(\widetilde{H}))={\mathcal{S}}({\mathrm{AJ}}_s(\widetilde{H}))\cap\Gamma^s_1= \{\Phi\in {\mathcal{S}}({\mathrm{AJ}}_s(\widetilde{H})):\pi^s_1\Phi=L_1\,\}.
$$
Since $\mathcal{S}_1({\mathrm{AJ}}_s(\widetilde{H}))=\ker\pi^s_1|_{\mathcal{S}({\mathrm{AJ}}_s(\tilde{H}))}$, so the family $\mathcal{S}_1({\mathrm{AJ}}_s(\widetilde{H}))$ is a normal divisor in $\mathcal{S}({\mathrm{AJ}}_s(\widetilde{H}))$. Furthermore, ${\mathcal{L}}_l\cap{\mathcal{S}}_1({\mathrm{AJ}}_s(\widetilde{H}))=\{L_1\}$ and if $\Phi\in{\mathcal{S}}({\mathrm{AJ}}_s(\widetilde{H}))$ with $\pi^s_1\Phi=L_c\in{\mathcal{L}}_l$ then $\Phi=L_c\circ\left(L_{c^{-1}}\circ \Phi\right)$ and $L_{c^{-1}}\circ \Phi\in{\mathcal{S}}_1({\mathrm{AJ}}_s(\widetilde{H}))$, hence ${\mathcal{S}}({\mathrm{AJ}}_s(\widetilde{H}))= {\mathcal{L}}_l\overline{\circledcirc}{\mathcal{S}}_1({\mathrm{AJ}}_s(\widetilde{H}))$ is a semidirect product (with respect to the substitution) of the subgroup ${\mathcal{L}}_l$ and the normal divisor ${\mathcal{S}}_1({\mathrm{AJ}}_s(\widetilde{H}))$. We have thus

\begin{theorem}\label{tNAJ}
Fix $s\in\mathbb{N}\cup\{\infty\}$, $l\in|1,s-1|$ and $H(x)=x^{l+1}+\sum_{i=l+2}^sh_ix^i\in\ps{k}{x}_s$. Let $T\in\Gamma^s_1$ transforms~$({\mathrm{AJ}}_s(H))$ to~$({\mathrm{AJ}}_s(\widetilde{H}))$, where $\widetilde{H}(x)=x^{l+1}+\delta x^{2l+1}$, $\delta\in\boldsymbol{k}$ and $\delta=0$ provided $s\leq 2l$. Then
\begin{enumerate}[(i)]
\item ${\mathcal{A}}_T({\mathcal{S}}({\mathrm{AJ}}_s(H)))= {\mathcal{S}}({\mathrm{AJ}}_s(\widetilde{H}))={\mathcal{L}}_l\overline{\circledcirc} {\mathcal{S}}_1(\mathrm{AJ}_s(\widetilde{H}))$,
\item
$$
\begin{array}{rcl}
{\mathcal{S}}({\mathrm{AJ}}_s(H))&=&{\mathcal{A}}_{T^{-1}}( {\mathcal{L}}_l\overline{\circledcirc} {\mathcal{S}}_1(\mathrm{AJ}_s(\widetilde{H})))\\[1ex]
&=&\left\{ T\circ \left(L_c\circ\Phi\right)\circ T^{-1}:c\in\boldsymbol{E}_l,\Phi\in {\mathcal{S}}_1({\mathrm{AJ}}_s(\widetilde{H})\,\right\}.
\end{array}
$$
\end{enumerate}
\end{theorem}

We have thus to study the algebraic structure of the group ${\mathcal{S}}_1({\mathrm{AJ}}_s(\widetilde{H}))$. Let us fix $s\in\mathbb{N}\cup\{\infty\}$, $l\in|1,s-1|$ and $\widetilde{H}(x)=x^{l+1}+\delta x^{2l+1}$, where $\delta\in\boldsymbol{k}$ and $\delta=0$ for $l+1\leq s\leq2l$. We begin with a technical result which allow us to prove the structure of the group~${\mathcal{S}}_1({\mathrm{AJ}}_s(\widetilde{H}))$.

\begin{lemma}\label{lx1}
Fix $s\in\mathbb{N}\cup\{\infty\}$ and $l\in|1,s-1|$. If for some $m\in|2,s|$ and $c_m\in\boldsymbol{k}^\star$ we have $\Phi(x)=x+c_mx^m+\sum_{i=m+1}^sc_ix^i\in {\mathcal{S}}_1({\mathrm{AJ}}_s(\widetilde{H}))$, then either $m=l+1$ or $s\in\mathbb{N}$ and $m\geq s-l+1$. In the last case $\Psi(x)=x+\sum_{i=s-l+1}^sc_ix^i\in{\mathcal{S}}_1(\mathrm{AJ}_s(\widetilde{H}))$ for all $c_{s-l+1},\ldots,c_s\in\boldsymbol{k}$.
\end{lemma}

\begin{proof}
Assume that $\Phi(x)=x+c_mx^m+\sum_{i=m+1}^sc_ix^i\in{\mathcal{S}}_1({\mathrm{AJ}}_s(\widetilde{H}))$ with $c_m\in\boldsymbol{k}^\star$. Rearranging $(\Phi)^{l+1}$ and $(\Phi)^{2l+1}$ with respect to subsequent powers of~$x$ we get
$$
\begin{array}{rcl}
\displaystyle \left(\Phi(x))\right)^{l+1}&=& x^{l+1}+(l+1)c_mx^{m+l}+\ldots \\[.5ex]
\displaystyle \delta\left(\Phi(x))\right)^{2l+1}&=&\displaystyle \delta\left(x^{2l+1}+(2l+1)c_mx^{m+2l}+\ldots\right)
\end{array}
$$
and
$$
\begin{array}{lll}
\lefteqn{\frac{\dd \Phi}{\dd x}(x^{l+1}+\delta x^{2l+1})=x^{l+1}+mc_mx^{m+l}+\sum_{i=m+1}^sic_ix^{i+l}}\null\\
& \qquad\qquad\qquad\qquad & \null\displaystyle +\delta\left(x^{2l+1}+mc_mx^{m+2l}+\sum_{i=m+1}^sic_ix^{i+2l}\right).
\end{array}
$$
Let us compare both sides of~$({\mathrm{AJ}}_s(\widetilde{H}))$. Assume first that $s=\infty$. Comparing coefficients for $x^{m+l}$ we get $(l+1)c_m=mc_m$, what jointly with $c_m\neq0$ implies $m=l+1$. Finally, if $s$ is finite then we have two possibilities: either $m+l\leq s$ and then also $m=l+1$, or $m+l\geq s+1$, so $m\geq s-l+1$. In the lase case we have thus $\Psi(x)=x+\sum_{i=s-l+1}^sc_ix^i\in{\mathcal{S}}_1(\mathrm{AJ}_s(\widetilde{H}))$ for all $c_{s-l+1},\ldots,c_s\in\boldsymbol{k}$.
\end{proof}

Using Lemmas~~\ref{lx2} and~\ref{lx1} we are able to prove structural properties of ${\mathcal{S}}_1({\mathrm{AJ}}_s(\widetilde{H}))$.

\begin{lemma}\label{lxx1}
If $s,l\in\mathbb{N}$ with $l+1\leq s$, then we have
\begin{enumerate}[(i)]
\item $\pi^{\infty}_s{\mathcal{S}}_1(\mathrm{AJ}_{\infty}(\widetilde{H}))\subset {\mathcal{S}}_1(\mathrm{AJ}_s(\widetilde{H}))$,
\item ${\mathcal{S}}_1^l({\mathrm{AJ}}_s(\widetilde{H}))= \left\{x+\sum\limits_{i=s-l+1}^sc_ix^i:c_{s-l+1},\ldots,c_s \in\boldsymbol{k}\,\right\}\subset {\mathcal{S}}_1({\mathrm{AJ}}_s(\widetilde{H}))$,
\item $\pi^{\infty}_s{\mathcal{S}}_1(\mathrm{AJ}_{\infty}(\widetilde{H}))\cap {\mathcal{S}}_1^l({\mathrm{AJ}}_s(\widetilde{H}))=\{L_1\}$ for $s\geq2l+1$,
 \item $\pi^{\infty}_s{\mathcal{S}}_1(\mathrm{AJ}_{\infty}(\widetilde{H})) \subset{\mathcal{S}}_1^l({\mathrm{AJ}}_s(\widetilde{H}))$ and ${\mathcal{S}}_1(\mathrm{AJ}_s(\widetilde{H}))={\mathcal{S}}_1^l ({\mathrm{AJ}}_s(\widetilde{H}))$ for $s\leq2l$,
\item ${\mathcal{S}}_1^l({\mathrm{AJ}}_s(\widetilde{H}))$ is a commutative normal divisor in~${\mathcal{S}}_1({\mathrm{AJ}}_s(\widetilde{H}))$ which is  isomorphic to the additive group $\boldsymbol{k}^l$,
\item if $s\geq2l+1$, $\Phi(x)=x+c_{l+1}x^{l+1}+\sum\limits_{i=l+2}^sc_ix^i\in \pi^{\infty}_s{\mathcal{S}}_1(\mathrm{AJ}_{\infty}(\widetilde{H}))$ and $\Psi(x)=x+\sum\limits_{i=s-l+1}^sd_ix^i\in{\mathcal{S}}_1^l({\mathrm{AJ}}_s(\widetilde{H}))$ then
$$
(\Phi\circ\Psi)(x)=(\Psi\circ \Phi)(x)=x+\sum_{i=l+1}^{s-l}c_ix^i+\sum_{i=s-l+1}^s(c_i+d_i)x^i.
$$
\end{enumerate}
\end{lemma}
\begin{proof} The inclusion {\em (i)} is an obvious consequence of Lemma~\ref{lx2} whereas {\em (ii)} follows from Lemma~\ref{lx1}. The properties {\em (iii)} and {\em (iv)} are consequences of Lemma~\ref{lx1} and an observation that $s-l+1\geq l+2$ if and only if $s\geq2l+1$. Indeed, if for some $c_m\in\boldsymbol{k}^\star$ we have  $\Phi(x)=x+c_mx^m+\sum_{i=m+1}^{\infty}\in{\mathcal{S}}_1({\mathrm{AJ}}_{\infty}(\widetilde{H}))$, then $m=l+1$ by Lemma~\ref{lx1}. Hence $\pi^{\infty}_s\Phi(x)=x+c_{l+1}x^{l+1}+\sum_{i=l+2}^s\notin {\mathcal{S}}_1^l({\mathrm{AJ}}_s(\widetilde{H}))$ for $s\geq2l+1$ and $\pi^{\infty}_s\Phi(x)=x+c_{l+1}x^{l+1}+\sum_{i=l+2}^s\in {\mathcal{S}}_1^l({\mathrm{AJ}}_s(\widetilde{H}))$ provided $s\leq2l$. Further,
$$
(\Psi_1\circ \Psi_2)(x)=x+\sum_{i=s-l+1}^s(c_i+d_i)x^i \quad \mbox{ and }\quad \Psi_1^{\circ-1}(x)=x+\sum_{i=s-l+1}^s(-c_i)x^i,
$$
for $\Psi_1(x)=x+\sum_{i=s-l+2}^sc_ix^i,\Psi_2(x)=x+\sum_{i=s-l+2}^sd_ix^i \in{\mathcal{S}}_1^l(\mathrm{AJ}_s(\widetilde{H}))$,  hence ${\mathcal{S}}_1^l(\mathrm{AJ}_s(\widetilde{H}))$ is isomorphic to the additive group $\boldsymbol{k}^l$. In the case $s\geq2l+1$, for  $\pi^s_{l+1}: {\mathcal{S}}_1(\mathrm{AJ}_s(\widetilde{H}))\to\Gamma^{l,l}_1$ we get ${\mathcal{S}}_1^l({\mathrm{AJ}}_s(\widetilde{H}))= \ker\pi^s_{l+1}|_{{\mathcal{S}}_1(\mathrm{AJ}_s(\tilde{H}))}$ by Lemma~\ref{lx1} whereas ${\mathcal{S}}_1(\mathrm{AJ}_s(\widetilde{H}))= {\mathcal{S}}_1^l({\mathrm{AJ}}_s(\widetilde{H}))$ for $s\leq2l$ (also from Lemma~\ref{lx1}), so in both cases ${\mathcal{S}}_1^l({\mathrm{AJ}}_s(\widetilde{H}))$ is a~normal divisor in ${\mathcal{S}}_1(\mathrm{AJ}_s(\widetilde{H}))$.

Finally, for $s\geq2l+1$ and $\Phi(x)=x+c_{l+1}x^{l+1}+\sum_{i=l+2}^sc_ix^i\in \pi^{\infty}_s{\mathcal{S}}_1({\mathrm{AJ}}_{\infty}(\widetilde{H}))$, $\Psi(x)=x+\sum_{i=s-l+1}^sd_ix^i\in{\mathcal{S}}_1^l ({\mathrm{AJ}}_s(\widetilde{H}))$ we have
$$
(\Phi\circ \Psi)(x)= x+\sum_{i=s-l+1}^sd_ix^i+\sum_{j=l+1}^sc_j\left( x+\sum_{i=s-l+1}^sd_ix^i\right)^j\mod x^{s+1},
$$
so $(\Phi\circ \Psi)(x)=x+\sum\limits_{i=l+1}^{s-l}c_ix^i+\sum_{i=s-l+1}^s(c_i+d_i)x^i$. On the other hand,
$$
(\Psi\circ \Phi)(x)= x+\sum_{i=l+1}^sc_ix^i+\sum_{j=s-l+1}^sd_j\left( x+\sum_{i=l+1}^sc_ix^i\right)^j\mod x^{s+1},
$$
hence $(\Psi\circ \Phi)(x)=x+\sum\limits_{i=l+1}^{s-l}c_ix^i+\sum_{i=s-l+1}^s(c_i+d_i)x^i$. This finishes the proof.
\end{proof}

We finish the study of this case with the following corollary.

\begin{corollary}\label{t3}
Fix $s\in\mathbb{N}$ and $l\in|1,s-1|$. Let $\widetilde{H}(x)=x^{l+1}+\delta x^{2l+1}$, where $\delta\in\boldsymbol{k}$ and $\delta=0$ whenever  $l+1\leq s\leq2l$.
\begin{enumerate}[(i)]
\item If $\pi^{\infty}_{l+1}: {\mathcal{S}}_1({\mathrm{AJ}}_{\infty}(\widetilde{H}))\to \Gamma^{l,l}_1$ is an isomorphism then
    $$
    {\mathcal{S}}_1({\mathrm{AJ}}_s(\widetilde{H}))\cong\left(\pi^{\infty}_s {\mathcal{S}}_1({\mathrm{AJ}}_{\infty}(\widetilde{H}))\right)\circledcirc {\mathcal{S}}_1^l({\mathrm{AJ}}_s(\widetilde{H}))\cong \boldsymbol{k}\times\boldsymbol{k}^l=\boldsymbol{k}^{l+1}
    $$
    provided $s\geq 2l+1$,
\item ${\mathcal{S}}_1({\mathrm{AJ}}_s(\widetilde{H}))= {\mathcal{S}}_1^l({\mathrm{AJ}}_s(\widetilde{H}))\cong \boldsymbol{k}^l$ for $s\leq2l$.
\end{enumerate}
\end{corollary}

\begin{proof}
The condition {\em (ii)} follows from Lemma~\ref{lxx1}~{\em (iv)}. Let $2l+1\leq s<\infty$. If $\pi^{\infty}_{l+1}: {\mathcal{S}}_1({\mathrm{AJ}}_{\infty}(\widetilde{H}))\to \Gamma^{l,l}_1$ is an isomorphism then $\pi^s_{l+1}:\pi^{\infty}_s({\mathcal{S}}_1({\mathrm{AJ}}_{\infty} (\widetilde{H}))) \to\Gamma^{l,l}_1$ is also an isomorphism. Clearly $\pi^{\infty}_s({\mathcal{S}}_1({\mathrm{AJ}}_{\infty}(\widetilde{H}))) \subset {\mathcal{S}}_1({\mathrm{AJ}}_{s}(\widetilde{H}))$ by Lemma~\ref{lxx1}~{\em(i)}, hence $\pi^s_{l+1}: {\mathcal{S}}_1({\mathrm{AJ}}_s(\widetilde{H}))\to\Gamma^{l,l}_1$ is an epimorphism. Define (an isomorphism)
$$
\Psi:=\left( \pi^s_{l+1}|_{\pi^{\infty}_s({\mathcal{S}}_1 ({\mathrm{AJ}}_{\infty}(\tilde{H})))} \right)^{-1}:\Gamma^{l,l}_1\to \pi^{\infty}_s({\mathcal{S}}_1({\mathrm{AJ}}_{\infty}(\tilde{H}))).
$$
Then $\Psi\circ\pi^s_{l+1}|_{{\mathcal{S}}_1({\mathrm{AJ}}_s(\widetilde{H}))}: {\mathcal{S}}_1({\mathrm{AJ}}_s(\widetilde{H}))\to \pi^{\infty}_s{\mathcal{S}}_1({\mathrm{AJ}}_{\infty}(\widetilde{H}))$ is an epimorphism with $\Psi\circ\pi^s_{l+1}|_{\pi^{\infty}_s{\mathcal{S}}_{\infty} ({\mathrm{AJ}}_s(\widetilde{H}))} =\id_{{\mathcal{S}}_1({\mathrm{AJ}}_{\infty}(\widetilde{H}))}$. Moreover, $\ker\pi^s_{l+1}|_{{\mathcal{S}}_1({\mathrm{AJ}}_s(\widetilde{H}))} ={\mathcal{S}}^l_1(\mathrm{AJ}_s(\widetilde{H}))$ and $\Psi$ is an isomorphism, whence  $\ker\Psi\circ\pi^s_{l+1}|_{{\mathcal{S}}_1({\mathrm{AJ}}_s(\widetilde{H}))} ={\mathcal{S}}^l_1(\mathrm{AJ}_s(\widetilde{H}))$. This jointly with Lemma~\ref{lxy1} proves {\em (i)}.
\end{proof}

\section{The Acz{\'e}l-Jabotynski formal differential equation -- the case $l=0$}

In Theorem~\ref{tNAJ1} we show the general solution of the Acz{\'e}l-Jabotynski formal differential equation in the case $l=0$. Unfortunately, the form of the solution give no dependence of solutions on the generator $H(x)=x+\sum_{j=1}^sh_jx^j$. We give now another form of solutions with indicated direct dependence of them on $h_j$'s.

\begin{theorem}[{cf.~\cite[Theorem 1]{Reich1}}]\label{t1}
Fix $H(x)=x+\sum_{n=2}^sh_nx^n\in\Gamma_1^s$. There exists a
sequence $(K_n)_{n\in|2,s|}$,
$K_n(y,t_2\ldots,t_n)\in\boldsymbol{q}[y;t_2,\ldots,t_n]$ of universal polynomials given recurrently by
$$
\begin{array}{lll}
\lefteqn{K_n(y,t_2,\ldots,t_n)=\frac{1}{n-1}
\left(t_n\sum_{j=0}^{n-2}y^j+\right.}\null\\
& & \null\displaystyle\sum_{i=2}^{n-1}t_i\sum_{\overline{u}_n\in
U_{n,i}}B_{\overline{u}_n}y^{u_1}(y^2-y)^{i-u_1-1}
\prod_{j=2}^{n-i+1}K_j(y,t_1,\ldots,t_j)^{u_j}\\
& & \null\displaystyle\;\;\left.-\sum_{i=2}^{n-1}t_i
(n-i+1)K_{n-i+1}(y,t_2\ldots,t_{n-i+1})\right)\;\;\;\mbox{ for }\,n\in|2,s|.
\end{array}
$$
such that to each $c_1\in\boldsymbol{k}^\star$
\begin{equation}\label{13}
\Phi_{c_1}(x)=c_1x+\sum_{n=2}^s(c_1^2-c_1)
K_n(c_1;h_2,\ldots,h_n)x^n
\end{equation}
is a unique solution of~$({\mathrm{AJ}}_s(H))$. The family ${\mathcal{S}}({\mathrm{AJ}}_s(H))$ of all
solutions~$\Phi_{c_1}$ of the equation~$({\mathrm{AJ}}_s(H))$ is a~commutative group under substitution isomorphic to $(\boldsymbol{k}^\star,\cdot)$.
\end{theorem}
\begin{proof}
We know that $c_1\in\boldsymbol{k}^\star$ is arbitrary and, from~\eqref{11} we get $c_2=h_2(c_1^2-c_1)$ and $c_3=(c_1^2-c_1)\left(\left(
\frac{h_3}{2}+h_2^2\right)(c_1+1)-2h_2^2\right)$. Define polynomials $K_2\in\boldsymbol{q}[y;t_2]$, $K_3\in\boldsymbol{q}[y;t_2,t_3]$ by
$K_2(y,t_2)=t_2$ and $K_3(y,t_2,t_3)=\left( \frac{t_3}{2}+t_2^2\right)(y+1)-2t_2^2$.

Let us fix $n\in|3,s|$ and assume that there exist  polynomials $K_j(y,t_2,\ldots,t_j)\in\boldsymbol{q}[y;t_2,\ldots,t_j]$ such that $c_j=(c_1^2-c_1)K_j(c_1,h_2,\ldots,h_j)$ for $j\in|2,n-1|$. Then
$$
\begin{array}{lll}
\lefteqn{\sum_{i=2}^{n-1}h_i\left(\sum_{\overline{u}_n\in
U_{n,i}}B_{\overline{u}_n}c_1^{u_1}\prod_{j=2}^{n-i+1}c_j^{u_j}-
(n-i+1)c_{n-i+1}\right)}\null  \\
& & \null\displaystyle\quad
=\sum_{i=2}^{n-1}h_i\sum_{\overline{u}_n\in
U_{n,i}}B_{\overline{u}_n}c_1^{u_1}
\prod_{j=2}^{n-i+1}\left((c_1^2-c_1)K_j(c_1,h_2,\ldots,h_j)\right)^{u_j}\\
& & \null\displaystyle\qquad\qquad-\sum_{i=2}^{n-1}(n-i+1)h_i
(c_1^2-c_1)K_{n-i+1}(c_1,h_2,\ldots,h_{n-i+1})\\
& & \null\displaystyle\quad
=\sum_{i=2}^{n-1}h_i\sum_{\overline{u}_n\in
U_{n,i}}B_{\overline{u}_n}c_1^{u_1}(c_1^2-c_1)^{i-u_1}
\prod_{j=2}^{n-i+1}K_j(c_1,h_2,\ldots,h_j)^{u_j}\\
& & \null\displaystyle\qquad\qquad-\sum_{i=2}^{n-1}(n-i+1)h_i(c_1^2-c_1)K_{n-i+1}(c_1,h_2,\ldots,h_{n-i+1}).
\end{array}
$$
Consequently, from~\eqref{11} with $c_n$ on the left hand side we obtain
$$
\begin{array}{lll}
\lefteqn{(n-1)c_n=h_n(c_1^n-c_1)}\null\\
& & \null\displaystyle+\sum_{i=2}^{n-1}h_i\sum_{\overline{u}_n\in
U_{n,i}}B_{\overline{u}_n}c_1^{u_1}(c_1^2-c_1)^{i-u_1}
\prod_{j=2}^{n-i+1}K_j(c_1,h_2,\ldots,h_j)^{u_j}\\
& & \null\displaystyle\qquad\qquad-\sum_{i=2}^{n-1}(n-i+1)h_i(c_1^2-c_1)K_{n-i+1}(c_1,h_2,\ldots,h_{n-i+1}),
\end{array}
$$
and next, on account of Lemma~\ref{l3}~{\em(v)} we infer
$$
\begin{array}{lll}
\lefteqn{(n-1)c_n=(c_1^2-c_1)\left(h_n\sum_{j=0}^{n-2}c_1^j+\right.}\null\\
& & \null\displaystyle\quad \sum_{i=2}^{n-1}h_i\sum_{\overline{u}_n\in
U_{n,i}}B_{\overline{u}_n}c_1^{u_1}(c_1^2-c_1)^{i-u_1-1}
\prod_{j=2}^{n-i+1}K_j(c_1,h_2,\ldots,h_j)^{u_j}\\
& & \null\displaystyle\qquad\qquad\left.-\sum_{i=2}^{n-1}(n-i+1)h_iK_{n-i+1}(c_1,h_2,\ldots,h_{n-i+1})\right).
\end{array}
$$
Define $K_n(y,t_2,\ldots,t_n)\in\boldsymbol{q}[y;t_2,\ldots,t_n]$ by
$$
\begin{array}{lll}
\lefteqn{K_n(y,t_2,\ldots,t_n)=\frac{1}{n-1}\left( t_n\sum_{j=0}^{n-2}y^j+\right.}\null\\
& & \null\displaystyle +
\sum_{i=2}^{n-1}t_i\sum_{\overline{u}_n\in
U_{n,i}}B_{\overline{u}_n}y^{u_1}(y^2-y)^{i-u_1-1}
\prod_{j=2}^{n-i+1}K_j(y,t_2,\ldots,t_j)^{u_j}\\
& & \null\displaystyle\qquad\qquad\left.-\sum_{i=2}^{n-1}(n-i+1)t_iK_{n-i+1}(y,t_2,\ldots,t_{n-i+1})\right).
\end{array}
$$
Clearly we have then
\begin{equation}\label{14}
c_n=(c_1^2-c_1)K_n(c_1,h_2,\ldots,h_n).
\end{equation}
Conversely, if $c_1\in\boldsymbol{k}^\star$ is arbitrarily fixed and $c_n$ for $n\in|2,s|$ are given by~\eqref{14}, then the system of equalities~\eqref{11} is satisfied, so $\Phi_{c_1}$  given by~\eqref{13} is a solution of~$({\mathrm{AJ}}_s(H))$.

Finally, by Fact~\ref{f1} we know that for a fixed $H$ the set  ${\mathcal{S}}({\mathrm{AJ}}_s(H))$ of solutions of~$({\mathrm{AJ}}_s(H))$ is a group. By what has been proved, we infer that for the canonical epimorphism $\pi^s_1$ we have that $\pi^s_1|_{{\mathcal{S}}({\mathrm{AJ}}_s(H))}:  {\mathcal{S}}({\mathrm{AJ}}_s(H))\to \Gamma^1$ is an isomorphism. Hence  ${\mathcal{S}}({\mathrm{AJ}}_s(H))\cong\Gamma_1\cong\boldsymbol{k}^\star$.
\end{proof}

\section{The Aczel-Jabotynski formal differential equation -- case $l\geq1$}

We consider now the group ${\mathcal{S}}({\mathrm{AJ}}_s(\widetilde{H})$ of all solutions of the tansformed Acz{\'e}l-Jabotinsky deifferential equation~$({\mathrm{AJ}}_s(\widetilde{H})$, where $s\in\mathbb{N}\cup\{\infty\}$, $l\in|1,s-1|$ are fixed and $\widetilde{H}(x)=x^{l+1}+\delta x^{2l+1}$, where $\delta\in\boldsymbol{k}$ and $\delta=0$ provided $s\leq2l$. By Theorem~\ref{tNAJ} we have ${\mathcal{S}}({\mathrm{AJ}}_s(\widetilde{H})={\mathcal{L}}_l \overline{\circledcirc}{\mathcal{S}}_1({\mathrm{AJ}}_s(\widetilde{H})$. In the description of ${\mathcal{S}}_1({\mathrm{AJ}}_s(\widetilde{H}))$ for finite~$s$ we use the group ${\mathcal{S}}_1({\mathrm{AJ}}_{\infty}(\widetilde{H}))$ and we focus now our attention on this group that is we are going now to determine all solutions  $\widetilde{\Phi}(x)=x+c_{l+1}x^{l+1}+\sum_{n=l+2}^{\infty} \widetilde{c}_nx^n\in{\mathcal{S}}_1({\mathrm{AJ}}_{\infty}(\widetilde{H})$ of the differential equation~$({\mathrm{AJ}}_{\infty}(\widetilde{H})$. We have
$$
\begin{array}{lll}
\lefteqn{\frac{\dd \widetilde{\Phi}}{\dd x}\cdot(x^{l+1}+\delta x^{2l+1})=
\left(
1+(l+1)\widetilde{c}_{l+1}x^l+\sum_{n=l+2}^{\infty}n\widetilde{c}_nx^{n-1}\right) \cdot(x^{l+1}+\delta x^{2l+1})}\null\\
& \qquad\qquad\qquad\qquad & \null\displaystyle
=x^{l+1}+((l+1)\widetilde{c}_{l+1}+\delta)x^{2l+1}+\sum_{n=2l+2}^{3l} (n-l)\widetilde{c}_{n-l}X^n \\
& &  \null\displaystyle \qquad\qquad\qquad + \sum_{n=3l+1}^{\infty}\left((n-l)\widetilde{c}_{n-l}+(n-2l)\delta
\widetilde{c}_{n-2l} \right)X^n.
\end{array}
$$
On the other hand (cf. Corollary~\ref{c1} with $m=l$ and~\eqref{xxy})
$$
\begin{array}{lll}
(\widetilde{\Phi}(x))^{l+1}+\delta(\widetilde{\Phi}(x))^{2l+1}\\
\;\displaystyle = \left(x+\widetilde{c}_{l+1}x^{l+1}+\sum_{n=l+2}^{\infty}\widetilde{c}_nx^n \right)^{l+1}+\delta\left(
x+\widetilde{c}_{l+1}x^{l+1}+\sum_{n=l+2}^{\infty}\widetilde{c}_nx^n \right)^{2l+1}\\
\;\displaystyle =
x^{l+1}+((l+1)\widetilde{c}_{l+1}+\delta)x^{2l+1}+\sum_{n=2l+2}^{3l} \left(l\widetilde{c}_{n-l}
+\sum_{\overline{u}_n\in\widetilde{U}^{l+1}_{n,l+1}} B_{\overline{u}_n}\prod_{j=l+1}^{n-l-1}
\widetilde{c}_j^{u_j}\right)x^n\\
\;\null\displaystyle \;\;
+\sum_{n=3l+1}^{\infty}\left((l+1)\widetilde{c}_{n-l}+ \sum_{\overline{u}_n\in\widetilde{U}^{l+1}_{n,l+1}} B_{\overline{u}_n}\prod_{j=l+1}^{n-l-1}
\widetilde{c}_j^{u_j}
+\delta\sum_{\overline{u}_n\in \widetilde{U}^{l+1}_{n,2l+1}}B_{\overline{u}_n}
\prod_{j=l+1}^{n-2l}\widetilde{c}_j^{u_j} \right)x^n.
\end{array}
$$
Hence, by comparing coefficient standing at subsequent powers of~$x$, from~$({\mathrm{AJ}}_{\infty}(\widetilde{H})$ we obtain $\widetilde{c}_l\in\boldsymbol{k}$ and the following system of equalities
$$
\begin{array}{lll}
\lefteqn{
(n-l)\widetilde{c}_{n-l}=(l+1)\widetilde{c}_{n-l} +\sum_{\overline{u}_n\in
\widetilde{U}^{l+1}_{n,l+1}}B_{\overline{u}_n}\prod_{j=l+1}^{n-l-1} \widetilde{c}_j^{u_j},
\quad n\in|2l+2,3l|}\null\\
\lefteqn{ \displaystyle
(n-l)\widetilde{c}_{n-l}=(l+1)\widetilde{c}_{n-l} +\sum_{\overline{u}_n\in
\widetilde{U}^{l+1}_{n,l+1}}B_{\overline{u}_n}\prod_{j=l+1}^{n-l-1} \widetilde{c}_j^{u_j}}\null\\
& & \null\displaystyle \quad +\delta \sum_{\overline{u}_n\in
\widetilde{U}^{l+1}_{n,2l+1}}B_{\overline{u}_n}\prod_{j=l+1}^{n-2l} \widetilde{c}_j^{u_j}
-(n-2l)\delta\widetilde{c}_{n-2l},\quad n\geq 3l+1,
\end{array}
$$
or, equivalently,
\begin{equation}\label{19}
\begin{array}{lll}
\lefteqn{
(n-2l-1)\widetilde{c}_{n-l}=\sum_{\overline{u}_n\in
\widetilde{U}^{l+1}_{n,l+1}}B_{\overline{u}_n}\prod_{j=l+1}^{n-l-1} \widetilde{c}_j^{u_j},
\quad n\in|2l+2,3l|}\null\\
\lefteqn{(n-2l-1)\widetilde{c}_{n-l}=\sum_{\overline{u}_n\in
\widetilde{U}^{l+1}_{n,l+1}}B_{\overline{u}_n}\prod_{j=l+1}^{n-l-1} \widetilde{c}_j^{u_j}}\null\\
& & \null\displaystyle +\delta \sum_{\overline{u}_n\in
\widetilde{U}^{l+1}_{n,2l+1}}B_{\overline{u}_n}\prod_{j=l+1}^{n-2l} \widetilde{c}_j^{u_j}
-(n-2l)\delta\widetilde{c}_{n-2l},\quad n\geq 3l+1.
\end{array}
\end{equation}

Now, we determine all solutions of~\eqref{19}. We prove

\begin{theorem}[{cf.~\cite[Theorem 2, Lemmas 2 and 3]{Reich1}}]\label{t2}
Fix $l\geq1$. There exists a~sequence of universal
polynomials $Q^{l+1}_n(y,t)\in\boldsymbol{q}[y,t]$ given recurrently by
\begin{equation}\label{seq4}
\begin{array}{lll}
\lefteqn{Q^{l+1}_{l+1}(y,t)=y,}\null\\
\lefteqn{Q^{l+1}_{rl+1}(y,t)=\frac{1}{(r-1)l}\Bigg(\sum_{\overline{u}_{n}\in
\widecheck{U}^{l+1}_{n,l+1}}B_{\overline{u}_n}
\prod_{j=1}^{r-1}Q^{l+1}_{jl+1}
(y,t)^{u_{jl+1}}}\null\\[1ex]
& & \null\displaystyle\qquad  +t
\sum_{\overline{u}_{n}\in
\widetilde{U}^{l+1}_{n,2l+1}}B_{\overline{u}_n}\prod_{j=1}^{r-1} Q^{l+1}_{jl+1}(y,t)^{u_{jl+1}}\\
& & \null\displaystyle\qquad 
-t((r-1)l+1)Q^{l+1}_{(r-1)l+1}(y,t) \Bigg), \mbox{ for }\,r\geq2\;\mbox{ with }\, n=(r+1)l+1,
\end{array}
\end{equation}
 such that to each $\widetilde{c}_{l+1}\in\boldsymbol{k}$
\begin{equation}\label{21}
\widetilde{\Phi}_{\widetilde{c}_{l+1}}(x)=x+\widetilde{c}_{l+1}x^{l+1}+
\sum_{p=2}^{\infty}Q^l_{pl+1}(\widetilde{c}_{l+1},\delta)x^{pl+1} \in{\mathcal{S}}_1({\mathrm{AJ}}_{\infty}(\widetilde{H})) \subset{\mathcal{N}}^{\infty,l}
\end{equation}
is a unique solution of~$({\mathrm{AJ}}_{\infty}(\widetilde{H}))$ in ${\mathcal{S}}_1({\mathrm{AJ}}_{\infty}(\widetilde{H}))$, where $\widetilde{H}(x)=x^{l+1}+\delta x^{2l+1}$. Then $\pi^{\infty}_{l+1}:({\mathrm{AJ}}_{\infty}(\widetilde{H}))\to \Gamma^{l,l}_1$ is an isomorphism, the group ${\mathcal{S}}_1({\mathrm{AJ}}_{\infty}(\widetilde{H}))$ is com\-mu\-tative and isomorphic to $(\boldsymbol{k},+)$.
\end{theorem}
\begin{proof}
Assume first that $l=1$. Then $|2l+2,3l|=\emptyset$ and $\widehat{U}^2_{n,3}=U_{n,3}$ so the system of equalities~\eqref{19}  we reduce to
\begin{equation}\label{22}
(n-3)\widetilde{c}_{n-1}=\sum_{\overline{u}_n\in
\overline{U}^2_{n,2}}B_{\overline{u}_n}\prod_{j=2}^{n-2}\widetilde{c}_j^{u_j} +\delta \sum_{\overline{u}_n\in U_{n,3}}B_{\overline{u}_n}\prod_{j=2}^{n-2}\widetilde{c}_j^{u_j}-\delta(n-2)\widetilde{c}_{n-2},\,n\geq4.
\end{equation}
Fix $\widetilde{c}_2$ arbitrarily. If we put $Q^2_2(y,t)=y$ then then $\widetilde{c}_2=Q_2^2(\widetilde{c}_2,\delta)$. For $n=4$ we have $\overline{U}_{4,2}=\{(0,2,0,0)\}$ and
$U_{4,3}=\{(2,1,0,0)\}$, so from~\eqref{22} we get
$$
\widetilde{c}_3=\frac{2!}{2!}\widetilde{c}_2^2+\delta\frac{3!}{2!}\widetilde{c}_2-\delta\cdot2\widetilde{c}_2=\widetilde{c}_2^2+\delta\widetilde{c}_2=:
Q^2_3(\widetilde{c}_2,\delta).
$$
Assuming that for some $n\geq4$ we have
$\widetilde{c}_{j}=Q^2_{j}(\widetilde{c}_2,\delta)$
for $j\in|2,n-2|$, from~\eqref{22} we obtain
$\widetilde{c}_{n-1}=Q^2_{n-1}(\widetilde{c}_2,\delta)$, where
\begin{equation}\label{516}
\begin{array}{lll}
\lefteqn{Q^2_{n-1}(y,t)=\frac{1}{n-3}\left(\sum_{\overline{u}_n\in
\overline{U}^2_{n,2}}B_{\overline{u}_n}\prod_{j=2}^{n-2} Q^2_j(y,t)^{u_j}\right.}\null\\
&\qquad\qquad & \null\displaystyle \left.+\delta\sum_{\overline{u}_n\in U_{n,3}}B_{\overline{u}_n}\prod_{j=2}^{n-2}Q^2_j(y,t)^{u_j}-\delta(n-2)Q^2_{n-2}(y,t)\rule{0pt}{8mm}\right).
\end{array}
\end{equation}
On this way we find $\widetilde{c}_n$ for all $n\geq 3$. Note, that~\eqref{516} is the same as~\eqref{seq4} with~$l=2$.

Fix $l\geq2$ now. Define $Q^{l+1}_{l+1}(y,t)=y$ and consider first~\eqref{19} for $n=2l+2$. Clearly $2l+2\notin\mathbb{N}_l$ and from Corollary~\ref{c3}~{\em (i)} we get
$\sum\limits_{\overline{u}_{2l+2}\in \widetilde{U}^{l+1}_{2l+2,l+1}}B_{\overline{u}_{2l+2}} \prod\limits_{j=l+1}^{l+1}\widetilde{c}_j^{u_j}=0$. Hence $\widetilde{c}_{l+2}=0$. If we assume that for some $n\in|2l+2,3l|$ (note that $n\notin\mathbb{N}_l$) we have $\widetilde{c}_j=0$ for $j\in|l+2,n-l-1|$ then by Corollary~\ref{c3}~{\em (i)} we obtain $\sum\limits_{\overline{u}_n\in
\widetilde{U}^{l+1}_{n,l+1}}B_{\overline{u}_n}\prod\limits_{j=l+1}^{n-l-1} \widetilde{c}_j^{u_j}=0$ which implies $\widetilde{c}_{n-l}=0$. Thus $\widetilde{c}_j=0$ for each $j\in|l+2,2l|$.

Let us consider the equality for $n=3l+1$ now. We proved that $\widetilde{c}_j=0$ for $j\in|l+2,2l|$. Further, we have
$$
\widetilde{U}^l_{3l+1,l+1}=\{l-1,0,\ldots,0,
\raise1.2pt\hbox{$\stackrel{l+1}{2}$},0,\ldots,0\}
\mbox{ and }
\widetilde{U}_{3l+1,2l+1}^l=\{2l,0,\ldots,0,
\raise1.2pt\hbox{$\stackrel{l+1}{1}$},0,\ldots,0\},
$$
hence
$$
\begin{array}{lll}
\lefteqn{l\widetilde{c}_{2l+1}=\sum_{\overline{u}_{3l+1}\in
\widetilde{U}_{3l+1,l+1}^{l+1}}B_{\overline{u}_{3l+1}} \prod_{j=l}^{2l}\widetilde{c}_j^{u_j}+\delta \sum_{\overline{u}_{3l+1}\in
\widetilde{U}_{3l+1,2l+1}^l}B_{\overline{u}_n}\prod_{j=l+1}^{l+1}
\widetilde{c}_j^{u_j}-(l+1)\delta\widetilde{c}_{l+1}}\null\\[1.5ex]
&\qquad\quad\;\; & \null\displaystyle=\frac{(l+1)!}{(l-1)!2!}\widetilde{c}_{l+1}^2+\delta
\frac{(2l+1)!}{(2l)!}\widetilde{c}_{l+1}-\delta
(l+1)\widetilde{c}_{l+1}=\frac{l(l+1)}{2}\widetilde{c}_{l+1}^2+l\delta \widetilde{c}_{l+1}.
\end{array}
$$
Thus $\widetilde{c}_{2l+1}=\frac{l+2}{2}\widetilde{c}_{l+1}^2+\delta \widetilde{c}_{l+1}$ and let us define $Q^l_{2l+1}(y,t):=\frac{l+1}{2}y^2+ty\in\boldsymbol{q}[y,t]$.

Finally, assume that for some $n\geq 3l+2$ we have
\begin{equation}\label{24}
\begin{array}{ll}
\widetilde{c}_j=0 & \mbox{ for }\, j\in|2,n-l-1|\setminus\mathbb{N}_l,\\
\widetilde{c}_{pl+1}=Q^{l+1}_{pl+1}(\widetilde{c}_{l+1};\delta) & \mbox{ for }\, p\in|0,r-1|,
\end{array}
\end{equation}
where $r\geq2$ is the last positive integer such that
$(r-1)l+1\leq n-l-1$. By our assumption~\eqref{24}, on account of Corollary~\ref{c3} we have
$$
\begin{array}{lll}
\lefteqn{\sum_{\overline{u}_n\in
\widetilde{U}^{l+1}_{n,l+1}}B_{\overline{u}_n}\prod_{j=l+1}^{n-l-1} \widetilde{c}_j^{u_j}
+\delta \sum_{\overline{u}_n\in
U^{l+1}_{n,2l+1}}B_{\overline{u}_n}\prod_{j=l+1}^{n-2l} \widetilde{c}_j^{u_j}
-(n-2l)\delta\widetilde{c}_{n-2l}}\null\\
& \qquad & \null\displaystyle =\sum_{\overline{u}_n\in \widecheck{U}^{l+1}_{n,l+1}}B_{\overline{u}_n}\prod_{j\in|l+1,n-l-1| \cap\mathbb{N}_l}
\widetilde{c}_j^{u_j}\\[1ex]
& & \null\displaystyle \qquad\qquad\qquad+\delta\sum_{\overline{u}_n\in U^{l+1}_{n,2l+1}}B_{\overline{u}_n}\prod_{j\in|l+1,n-2l|\cap\mathbb{N}_l} \widetilde{c}_j^{u_j}
 -(n-2l)\delta\widetilde{c}_{n-2l},
\end{array}
$$
consequently, we will consider the equality
\begin{equation}\label{25}
\begin{array}{lll}
\lefteqn{(n-1-2l)\widetilde{c}_{n-l}=\sum_{\overline{u}_n\in
\widecheck{U}_{n,l+1}^{l+1}}B_{\overline{u}_n}\prod_{j\in|1,n-l-1| \cap\mathbb{N}_l}
\widetilde{c}_j^{u_j}}\null\\[1ex]
&\qquad\qquad & \null \displaystyle +\delta
\sum_{\overline{u}_n\in
U_{n,2l+1}^{l+1}}B_{\overline{u}_n}\prod_{j\in|1,n-2l|\cap\mathbb{N}_l}
\widetilde{c}_j^{u_j}
 -(n-2l)\delta\widetilde{c}_{n-2l}.
\end{array}
\end{equation}
If $n\notin\mathbb{N}_l$, then similarly as above one can show that $\widetilde{c}_{n-l}=0$. If $n\in\mathbb{N}_l$, that is $n=(r+1)l+1$, then from~\eqref{25} we get
$$
\begin{array}{lll}
\lefteqn{(r-1)l\widetilde{c}_{rl+1} =
\sum_{\overline{u}_{n}\in
\widecheck{U}^{l+1}_{n,l+1}}B_{\overline{u}_{n}}
\widetilde{c}_1^{u_1} \prod_{j=1}^{r-1}
\widetilde{c}_{jl)+1}^{u_{jl+1}}}\null\\
& & \displaystyle +\delta \sum_{\overline{u}_{n}\in
U^{l+1}_{n,2l+1}}B_{\overline{u}_{n}}\widetilde{c}_1^{u_1} \prod_{j=1}^{r-1} \widetilde{c}_{jl+1}^{u_{jl+1}}
-((r-1)l+1)\delta\widetilde{c}_{(r-1)l+1}\;\mbox{ with }\,n=(r+1)l+1.
\end{array}
$$
 Thus we may define a polynomial $Q^{l+1}_{rl+1}(y,t)\in\boldsymbol{q}[y,t]$ as in~\eqref{seq4} such that
$\widetilde{c}_{rl+1}=Q^{l+1}_{rl+1}(\widetilde{c}_2,\delta)$. On this way we obtain $\widetilde{c}_j$ for every $j\geq l+1$.

We have proved, that to each $\widetilde{c}_l\in\boldsymbol{k}$ there exists a unique solution $\Phi_{\widetilde{c}_l}\in{\mathcal{S}}_1({\mathrm{AJ}}_{\infty}(\widetilde{H}))$. Consequently $\pi^{\infty}_{l+1}:{\mathcal{S}}_1({\mathrm{AJ}}_{\infty}(\widetilde{H})) \to\Gamma^{l,l}_1$ is an epimorphism, hence ${\mathcal{S}}_1({\mathrm{AJ}}_{\infty}(\widetilde{H}))$ is isomorphic to $(\boldsymbol{k},+)$ by Theorem~\ref{t3} and Lemma~\ref{lx3}~{\em (ii)}.
\end{proof}

We are able to give the desctiption of the group of solutions of the Aczel-Jabotinsky differential equation~$({\mathrm{AJ}}_{\infty}(H))$. From Theorems~\ref{tNAJ} and \ref{t2} we obtain

\begin{corollary}\label{c4}
Fix $l\in\mathbb{N}$ and $H(x)=x^{l+1}+\sum_{j=l+2}^{\infty}h_jx^j\in\ps{k}{x}$. Let $T(x)=x+\sum_{j=2}^{\infty}v_jx^j\in\Gamma_1$ transforms $({\mathrm{AJ}}_{\infty}(H))$ to~$({\mathrm{AJ}}_{\infty}(\widetilde{H}))$, where $\widetilde{H}(x)=x^{l+1}+\delta x^{2l+1}$ with some $\delta\in\boldsymbol{k}$. Define a~sequence of universal polynomials $Q^l_n(y,t)\in\boldsymbol{q}[y,t]$ by~\eqref{seq4}.
\begin{enumerate}[(i)]
\item To every $c_1\in\boldsymbol{E}_l$ and $c_{l+1}\in\boldsymbol{k}$
$$
\begin{array}{lll}
\lefteqn{\widehat{\Phi}_{c_1,c_l}=L_{c_1}\circ\widetilde{\Phi}_{c_1^{-1}c_l} =c_1x+c_{l+1}x^{l+1} }\null\\
&\qquad&\null\displaystyle +
\sum_{p=2}^{\infty}c_1Q^{l+1}_{pl+1}(c_1^{-1}c_{l+1},\delta) x^{pl+1}\in{\mathcal{S}}({\mathrm{AJ}}_{\infty}(\widetilde{H})) \subset{\mathcal{N}}^{\infty,l}
\end{array}
$$
is the unique solution of~$({\mathrm{AJ}}_{\infty}(\widetilde{H}))$, where $\widetilde{\Phi}_{\widetilde{c}_{l+1}}\in{\mathcal{S}}_1({\mathrm{AJ}}_{\infty}(\widetilde{H}))$ is a solution of~$({\mathrm{AJ}}_{\infty}(\widetilde{H}))$ described in Theorem~\ref{t2}. The group ${\mathcal{S}}({\mathrm{AJ}}_{\infty}(\widetilde{H}))$ of all solutions of~$({\mathrm{AJ}}_{\infty}(\widetilde{H}))$ is a direct product ${\mathcal{L}}_l\circledcirc {\mathcal{S}}_1({\mathrm{AJ}}_{\infty}(\widetilde{H}))$ of the subgroups ${\mathcal{L}}_l$ and ${\mathcal{S}}_1({\mathrm{AJ}}_{\infty}(\widetilde{H}))$ and it is then isomorphic to $\Gamma^{l,l}\cong(\boldsymbol{E}_l\times\boldsymbol{k})$.
\item To every $c_1\in\boldsymbol{E}_l$ and $c_{l+1}\in\boldsymbol{k}$
$$
\begin{array}{rcl}
\Phi_{c_1,c_{l+1}}&=&T\circ\widehat{\Phi}_{c_1,c_{l+1}}\circ T^{\circ-1}=T\circ\left(L_{c_1}\circ\widetilde{\Phi}_{c_1^{-1}c_l}\right)\circ T^{\circ-1}\\
&=&\left(T\circ L_{c_1}\circ T^{\circ-1}\right)\circ\left(T\circ \widetilde{\Phi}_{c_1^{-1}c_l}\circ T^{\circ-1}\right)
\end{array}
$$
is the unique solution of the Aczel-Jabotinsky differential equation~$({\mathrm{AJ}}_{\infty}(H))$, where $\widehat{\Phi}_{c_1^{-1}c_l}\in{\mathcal{S}}_1({\mathrm{AJ}}_{\infty}(\widetilde{H}))$ is a solution of~$({\mathrm{AJ}}_{\infty}(\widetilde{H}))$. Hence the group ${\mathcal{S}}({\mathrm{AJ}}_{\infty}(H))$ of all solutions of~$({\mathrm{AJ}}_{\infty}(H))$ is is a direct product ${\mathcal{A}}_T({\mathcal{L}}_l)\circledcirc {\mathcal{A}}_T({\mathcal{S}}_1({\mathrm{AJ}}_{\infty}(\widetilde{H})))$ of the subgroups ${\mathcal{A}}_T({\mathcal{L}}_l)$ and ${\mathcal{A}}_T({\mathcal{S}}_1({\mathrm{AJ}}_{\infty}(\widetilde{H})))$ and it is also isomorphic to $\Gamma^{l,l}\cong(\boldsymbol{E}_l\times\boldsymbol{k})$.
\end{enumerate}
\end{corollary}
\begin{proof}
Fix a solution $\widehat{\Phi}(x)=c_1x+\sum_{j=2}^{\infty}c_jx^j\in{\mathcal{S}}({\mathrm{AJ}}_{\infty}(\widetilde{H}))$. Then $c_1\in\boldsymbol{E}_l$ and from the proof of Theorem~\ref{tNAJ} we get $\widehat{\Phi}=L_{c_1}\circ\left(L_{c_1^{-1}}\circ\widehat{\Phi}\right)\in {\mathcal{L}}_l\overline{\circledcirc}{\mathcal{S}}_1({\mathrm{AJ}}_{\infty}(\widetilde{H}))$ and $(L_{c_1^{-1}}\circ\widehat{\Phi})(x)=x+\sum_{j=2}^{\infty}c_1^{-1}c_jx^j\in{\mathcal{S}}_1({\mathrm{AJ}}_{\infty}(\widetilde{H}))$. From Lemma~\ref{lx1} we obtain then $(L_{c_1^{-1}}\circ\widehat{\Phi})(x)=x+c_1^{-1}c_{l+1}x^{l+1} +\sum_{j=l+2}^{\infty}c_1^{-1}c_jx^j$. Theorem~\ref{t2} implies that $\pi^{\infty}_{l+1}|_{{\mathcal{S}}_1({\mathrm{AJ}}_{\infty}(\tilde{H}))}: {\mathcal{S}}_1({\mathrm{AJ}}_{\infty}(\widetilde{H}))\to\Gamma^{l,l}_1$ is an epimorphism and moreover
$$
(L_{c_1^{-1}}\circ\widehat{\Phi})(x)=x+c_1^{-1}c_{l+1}x^{l+1}+ \sum_{p=2}^{\infty}Q^{l+1}_{pl+1}(c_1^{-1}c_{l+1},\delta)x^{pl+1} \in{\mathcal{S}}_1({\mathrm{AJ}}_{\infty}(\widetilde{H})) \subset{\mathcal{N}}^{\infty,l}.
$$
Note that $L_{c_1}\in{\mathcal{L}}_l$ and $L_{c_1^{-1}}\circ\widehat{\Phi}\in{\mathcal{N}}^{\infty,l}$, so $L_{c_1}$ and $L_{c_1^{-1}}\circ\widehat{\Phi}$ commute. Hence both, ${\mathcal{L}}_l$ and ${\mathcal{S}}_1({\mathrm{AJ}}_{\infty}(\widetilde{H}))$ are normal divisors in ${\mathcal{S}}({\mathrm{AJ}}_{\infty}(\widetilde{H}))$, therefore ${\mathcal{S}}({\mathrm{AJ}}_{\infty}(\widetilde{H}))= {\mathcal{L}}_l\circledcirc{\mathcal{S}}_1({\mathrm{AJ}}_{\infty}(\widetilde{H}))$. Finally,
$$
\widehat{\Phi}(x)=\widehat{\Phi}_{c_1,c_l}(x)=c_1x+c_{l+1}x^{l+1}+ \sum_{p=2}^{\infty}c_1Q^{l+1}_{pl+1}(c_1^{-1}c_l,\delta)x^{pl+1},
$$
$\pi^{\infty}_s|_{{\mathcal{S}}({\mathrm{AJ}}_{\infty}(\tilde{H}))}: {\mathcal{S}}({\mathrm{AJ}}_{\infty}(\widetilde{H}))\to\Gamma^{l,l}$ is an epimorphism with $\ker\pi^{\infty}_s|_{{\mathcal{S}}({\mathrm{AJ}}_{\infty}(\tilde{H}))}= \{L_1\}$, so ${\mathcal{S}}({\mathrm{AJ}}_{\infty}(\widetilde{H})) \cong\Gamma^{l,l}\cong(\boldsymbol{E}_l\times\boldsymbol{k})$.

The statement~{\em (ii)} is then a consequence of~{\em (i)} and Theorem~\ref{tNAJ}~{\em (ii)}.
\end{proof}

We now consider the case $s\in\mathbb{N}$. The proof of the results we present here runns on a similar way to the proof of Corollary~\ref{c4} an it will be omitted.

\begin{corollary}\label{c5}
Fix $l\in\mathbb{N}$, $s\geq2l+1$ and $H(x)=x^{l+1}+\sum_{j=l+2}^sh_jx^j\in\ps{k}{x}_s$. Assume that $T(x)=x+\sum_{j=2}^sv_jx^j\in\Gamma^s_1$ transforms $({\mathrm{AJ}}_s(H))$ to~$({\mathrm{AJ}}_s(\widetilde{H}))$, $\widetilde{H}(x)=x^{l+1}+\delta x^{2l+1}$ with some $\delta\in\boldsymbol{k}$. Let $r\in\mathbb{N}$ be the last positive integer such that $rl+1\leq s$ and define a~sequence of  polynomials $Q^{l+1}_n(y,t)\in\boldsymbol{q}[y,t]$.
\begin{enumerate}[(i)]
\item To every $c_1\in\boldsymbol{E}_l$ and $c_{l+1},c_{s-l+1},\ldots,c_s\in\boldsymbol{k}$
$$
\begin{array}{lll}
\lefteqn{\widehat{\Phi}_{c_1,c_{l+1},c_{s-l+1},\ldots,c_s}=L_{c_1}\circ\left( \widetilde{\Phi}_{c_1^{-1}c_{l+1}}\circ \overline{\Phi}_{c_1^{-1}c_{s-l+1},\ldots,c_1^{-1}c_s}\right) }\null\\
& \qquad\quad &\null\displaystyle  =c_1x+c_{l+1}x^{l+1} +
\sum_{p=2}^{r-1}c_1Q^{l+1}_{pl+1}(c_1^{-1}c_{l+1},\delta)x^{pl+1}+ \sum_{j=s-l+1}^{rl}c_jx^j\\
& & \null\displaystyle\quad +\left(c_1Q^{l+1}_{rl+1}(c_1^{-1}c_{l+1},\delta)+c_{rl+1}\right)x^{rl+1}+ \sum_{j=rl+2}^sc_jx^j\in{\mathcal{S}}({\mathrm{AJ}}_s(\widetilde{H}))
\end{array}
$$
is the unique solution of~$({\mathrm{AJ}}_s(\widetilde{H}))$, where $\overline{\Phi}_{c_1^{-1}c_{s-l+1},\ldots,c_1^{-1}c_s}\in{\mathcal{S}}_1^l ({\mathrm{AJ}}_s(\widetilde{H}))$, $\widetilde{\Phi}_{\widetilde{c}_{l+1}}= \pi^{\infty}_s\widecheck{\Phi}_{\widetilde{c}_{l+1}}$ and $\widecheck{\Phi}_{\widetilde{c}_{l+1}}\in{\mathcal{S}}_1({\mathrm{AJ}}_{\infty}( \widetilde{H}))$ is a solution of~$({\mathrm{AJ}}_{\infty}(\widetilde{H}))$ described in Theorem~\ref{t2}. Then the group ${\mathcal{S}}({\mathrm{AJ}}_s(\widetilde{H}))$ of all solutions of~$({\mathrm{AJ}}_s(\widetilde{H}))$ is a semi-direct product ${\mathcal{L}}_l\overline{\circledcirc}\left(\pi^{\infty}_s{\mathcal{S}}_1 ({\mathrm{AJ}}_{\infty}(\widetilde{H}))\circledcirc{\mathcal{S}}_1^l ({\mathrm{AJ}}_s(\widetilde{H}))\right)$ of subgroups ${\mathcal{L}}_l$, $\pi^{\infty}_s{\mathcal{S}}_1({\mathrm{AJ}}_{\infty}(\widetilde{H}))$ and ${\mathcal{S}}_1^l({\mathrm{AJ}}_s(\widetilde{H}))$ and it is then isomorphic to $(\boldsymbol{E}_l\times\boldsymbol{k}^{l+1},\overline{\diamond})$.
\item To every $c_1\in\boldsymbol{E}_l$ and $c_{l+1},c_{s-l+1},\ldots,c_s\in\boldsymbol{k}$
$$
\begin{array}{rcl}
\Phi_{c_1,c_{l+1},c_{s-l+1},\ldots,c_s}&=&T\circ\widehat{\Phi}_{c_1,c_{l+1}, c_{s-l+1},\ldots,c_s}\circ T^{\circ-1}\\
&=&T\circ\left(L_{c_1}\circ\widetilde{\Phi}_{c_1^{-1}c_{l+1}} \circ \overline{\Phi}_{c_1^{-1}c_{s-l+1},\ldots,c_1^{-1}c_s}\right)\circ T^{\circ-1}\\
&=&\left(T\circ L_{c_1}\circ T^{\circ-1}\right)\circ\left(T\circ \widetilde{\Phi}_{c_1^{-1}c_{l+1}}\circ T^{\circ-1}\right)\\
& & \qquad\qquad\quad \circ\left( T\circ \overline{\Phi}_{c_1^{-1}c_{s-l+1},\ldots,c_1^{-1}c_s}\circ T^{\circ-1}\right)
\end{array}
$$
is the unique solution of the Aczel-Jabotinsky differential equation~$({\mathrm{AJ}}_s(H))$, where $\widetilde{\Phi}_{\widetilde{c}_{l+1}}=\pi^{\infty}_s\widecheck{\Phi}_{\widetilde{c}_{l+1}}$, $\widecheck{\Phi}_{\widetilde{c}_{l+1}}\in{\mathcal{S}}_1({\mathrm{AJ}}_{\infty}(\widetilde{H}))$ is a solution of~$({\mathrm{AJ}}_{\infty}(\widetilde{H}))$ and $\overline{\Phi}_{c_1^{-1}c_{s-l+1},\ldots,c_1^{-1}c_s}\in{\mathcal{S}}_1^l ({\mathrm{AJ}}_s(\widetilde{H}))$. Hence the group ${\mathcal{S}}({\mathrm{AJ}}_s(H))$ of all solutions of~$({\mathrm{AJ}}_s(H))$ is a semidirect product
$$
{\mathcal{A}}_T({\mathcal{L}}_l)\overline{\circledcirc}\left( {\mathcal{A}}_T(\pi^{\infty}_s{\mathcal{S}}_1({\mathrm{AJ}}_{\infty}( \widetilde{H})))\circledcirc{\mathcal{A}}_T({\mathcal{S}}_1^l ({\mathrm{AJ}}_s(\widetilde{H})))\right)
$$
of the subgroups ${\mathcal{A}}_T({\mathcal{L}}_l)$, ${\mathcal{A}}_T(\pi^{\infty}_s{\mathcal{S}}_1({\mathrm{AJ}}_{\infty}(\widetilde{H})))$ and ${\mathcal{A}}_T({\mathcal{S}}_1^l({\mathrm{AJ}}_s(\widetilde{H})))$ and it is also isomorphic to $(\boldsymbol{E}_l\times\boldsymbol{k}^{l+1},\overline{\diamond})$.
\end{enumerate}
\end{corollary}

\begin{corollary}\label{c6}
Fix $2\leq l+1\leq s\leq2l$ and $H(x)=x^{l+1}+\sum_{j=l+2}^sh_jx^j\in\ps{k}{x}_s$. Assume that $T(x)=x+\sum_{j=2}^sv_jx^j\in\Gamma^s_1$ transforms $({\mathrm{AJ}}_s(H))$ to~$({\mathrm{AJ}}_s(\widetilde{H}))$, where $\widetilde{H}(x)=x^{l+1}$.
\begin{enumerate}[(i)]
\item To every $c_1\in\boldsymbol{E}_l$ and $c_{s-l+1},\ldots,c_s\in\boldsymbol{k}$
$$
\widehat{\Phi}_{c_1,c_{s-l+1},\ldots,c_s}=L_{c_1}\circ \overline{\Phi}_{c_1^{-1}c_{s-l+1},\ldots,c_1^{-1}c_s}=c_1x+ \sum_{j=s-l+1}^sc_jx^j\in{\mathcal{S}}({\mathrm{AJ}}_s(\widetilde{H}))
$$
is the unique solution of~$({\mathrm{AJ}}_s(\widetilde{H}))$, where $\overline{\Phi}_{c_1^{-1}c_{s-l+1},\ldots,c_1^{-1}c_s}\in{\mathcal{S}}_1^l ({\mathrm{AJ}}_s(\widetilde{H}))$. The group ${\mathcal{S}}({\mathrm{AJ}}_s(\widetilde{H}))$ of all solutions of~$({\mathrm{AJ}}_s(\widetilde{H}))$ is a semidirect product ${\mathcal{L}}_l\overline{\circledcirc}{\mathcal{S}}_1^l ({\mathrm{AJ}}_s(\widetilde{H}))$ of the subgroups ${\mathcal{L}}_l$ and ${\mathcal{S}}_1^l({\mathrm{AJ}}_s(\widetilde{H}))$ and it is then isomorphic to $(\boldsymbol{E}_l\times\boldsymbol{k}^l,\widehat{\diamond})$.
\item To every $c_1\in\boldsymbol{E}_l$ and $c_{s-l+1},\ldots,c_s\in\boldsymbol{k}$
$$
\begin{array}{rcl}
\Phi_{c_1,c_{s-l+1},\ldots,c_s}&=&T\circ\widehat{\Phi}_{c_1,c_{l+1},c_{s-l+1}, \ldots,c_s}\circ T^{\circ-1}\\
&=&T\circ\left(L_{c_1}\circ \overline{\Phi}_{c_1^{-1}c_{s-l+1},\ldots,c_1^{-1}c_s}\right)\circ T^{\circ-1}\\
&=&\left(T\circ L_{c_1}\circ T^{\circ-1}\right)\circ\left( T\circ \overline{\Phi}_{c_1^{-1}c_{s-l+1},\ldots,c_1^{-1}c_s}\circ T^{\circ-1}\right)
\end{array}
$$
is the unique solution of the Aczel-Jabotinsky differential equation~$({\mathrm{AJ}}_s(H))$, where $\overline{\Phi}_{c_1^{-1}c_{s-l+1},\ldots,c_1^{-1}c_s}\in{\mathcal{S}}_1^l ({\mathrm{AJ}}_s(\widetilde{H}))$. Hence the group ${\mathcal{S}}({\mathrm{AJ}}_s(H))$ of all solutions of~$({\mathrm{AJ}}_s(H))$ is a semidirect product ${\mathcal{A}}_T({\mathcal{L}}_l)\overline{\circledcirc} {\mathcal{A}}_T({\mathcal{S}}_1^l({\mathrm{AJ}}_s(\widetilde{H})))$
of the subgroups ${\mathcal{A}}_T({\mathcal{L}}_l)$ and ${\mathcal{A}}_T({\mathcal{S}}_1^l({\mathrm{AJ}}_s(\widetilde{H})))$ and it is also isomorphic to $(\boldsymbol{E}_l\times\boldsymbol{k}^l,\widehat{\diamond})$.
\end{enumerate}
\end{corollary}

\section{Examples and problems}

One can observe that for $s\in\mathbb{N}$, $s\geq3$ and $l\in|2,s-1|$ the groups ${\mathcal{S}}({\mathrm{AJ}}_s(H))$ of all solutions of $({\mathrm{AJ}}_s(H))$ need not be commutative. Indeed, they are then isomorphic either to $(\boldsymbol{E}_l\times\boldsymbol{k}^{l+1},\overline{\diamond})$ for $s\geq2l+1$ or to $(\boldsymbol{E}_l\times\boldsymbol{k}^l,\widehat{\diamond})$ in the case $s\leq2l$ and these groups need not be commutative.

\begin{example}
Assume that $l+1=s=3$. Clearly $\boldsymbol{E}_2=\{-1,1\}$. Let us consider a family
$$
{\mathcal{F}} =\left\{ \Phi(x)=c_1x+c_2x^2+c_3x^3: c_1\in\boldsymbol{E}_2, c_2,c_3\in\boldsymbol{k}\,\right\}.
$$
We show that ${\mathcal{F}}={\mathcal{S}}({\mathrm{AJ}}_3(H))$ for $H(x)=x^3$. Indeed
$$
\begin{array}{ll}
\displaystyle
\left(\Phi(x)\right)^3=\left(c_1x+c_2x^2+c_3x^3\right)^3=c_1^3x^3 & \mod x^4,\\[1ex]
\displaystyle \frac{\dd \Phi}{\dd x}\cdot H(x)=\left(c_1+2c_2x+3c_2x^2\right)\cdot x^3=c_1x^3 & \mod x^4.
\end{array}
$$
Hence $\Phi(x)\in {\mathcal{S}}({\mathrm{AJ}}_3(H))$ and we see that there are no other solutions. Take now $\Phi_1(x)=c_1x+c_2x^2$, $\Phi_2(x)=-x$ for some $c_1\in\boldsymbol{E}_2$ and $c_2\in\boldsymbol{k}^*$. Then $\Phi_1,\Phi_2\in{\mathcal{S}}({\mathrm{AJ}}_3(H))$ and
$$
\begin{array}{ll}
(\Phi_1\circ \Phi_2)(x)=c_1(-x)+c_2(-x)^2=-c_1x+c_2x^2 & \mod x^4,\\[.5ex]
(\Phi_2\circ \Phi_1)(x)=-(c_1 x+c_2x^2)=-c_1x-c_2x^2 & \mod x^4.
\end{array}
$$
Thus $\Phi_1\circ\Phi_2\neq\Phi_2\circ\Phi_1$, so ${\mathcal{S}}({\mathrm{AJ}}_3(H))$ is not commutative.
\end{example}

\begin{example}
Fix $l,s\in\mathbb{N}$, $l\geq2$ and $l\leq s\leq 2l$. Then on account of Corollary~\ref{c6} the family
$$
{\mathcal{F}}={\mathcal{L}}_l\overline{\ocirc} {\mathcal{S}}_1^l({\mathrm{AJ}}_s(x^l))=\left\{c_1x+\sum_{j=s-l+1}^s c_1c_jx^j:c_1\in\boldsymbol{E}_l, c_{s-l+1},\ldots,c_s\in\boldsymbol{k}\,\right\}
$$
is a group of solutions of $({\mathrm{AJ}}_s(x^{l+1}))$. The group ${\mathcal{F}}$ is clearly isomorphic to the group $(\boldsymbol{E}_l\times\boldsymbol{k}^l,\widehat{\diamond})$ and it is not commutative provided $\{1\}\subsetneq\boldsymbol{E}_l$. Hence we get an example of noncommutative group being a semidirect product of commutative groups.
\end{example}

We construct now a general example. For fixed $l\geq1$ we put
$$
(kl+1)!_{l}:=\prod_{j=0}^k(jl+1).
$$
Define a family $(F^{(l)}_t)_{t\in\boldsymbol{k}}$,
\begin{equation}\label{genex}
F^{(l+1)}_t(x)=x+\sum_{n=1}^{\infty} \frac{((n-1)l+1)!_{l}}{n!}t^nx^{nl+1}\in\Gamma^s_1 \quad\mbox{ for }\,t\in\boldsymbol{k}.
\end{equation}

Applying techniques used in the proof of Theorem~1 in~\cite{Jabl2007a} one can easily prove the following fact.

\begin{fact}
Let $s\in\mathbb{N}\cup\{\infty\}$. The family $(F^{(2)})_{t\in\boldsymbol{k}}$
$$
F^{(2)}_t(x)=x+\sum_{n=2}^{\infty}t^{n-1}x^n \qquad\mbox{ for }\,t\in\boldsymbol{k}
$$
is a one-parameter group of formal power series which means that
$$
F^{(2)}_{t_1+t_2}=F^{(2)}_{t_1}\circ F^{(2)}_{t_2} \qquad\mbox{ for }\,t_1,t_2\in\boldsymbol{k}.
$$
Then
\begin{enumerate}[(i)]
\itemsep-2pt
\item $(F^{(2)}_t)_{t\in\boldsymbol{k}}$ is a commutative group of solutions of the Acz{\'e}l-Jabotinsky formal differential equation $({\mathrm{AJ}}_{\infty}(x^2))$ with the generator $H(x)=x^2$,
\item  for $s\in\mathbb{N}$, $s\geq2$, the family
$$
{\widehat{\mathcal{F}}}=\left\{ \widehat{F}^{(2)}_{t,c}(x)=x+\sum_{n=2}^{s-1}t^{n-1}x^n+\left( c+t^{s-1}\right)x^s:\; t,c\in\boldsymbol{k}\,\right\},
$$
is a commutative group of solutions of the Acz{\'e}l-Jabotinsky formal differential equation $({\mathrm{AJ}}_s(x^2))$ with the generator $H(x)=x^2$ for finite $s\geq3$.
\end{enumerate}
\end{fact}

\begin{remark}
We have thus construct explicit form of the group of solutions of the Acz{\'e}l-Jabotinsky formal differential equation for the generator $H(x)=x^2$.
\end{remark}

\begin{problem}
Is it possible to construct an explicit form of the group of solutions of the Acz{\'e}l-Jabotinsky formal differential equation for the generator $\widetilde{H}(x)=x^{l+1}+\delta x^{2l+1}$, where $\delta\in\boldsymbol{k}^*$ and $l\geq1$?
\end{problem}

\begin{problem}
Is the family $(F^{(l+1)}_t)_{t\in\boldsymbol{k}}$ defined by~\eqref{genex} a one-parameter group of formal power series for arbitrary $l\geq2$?
\end{problem}

\begin{remark}
In the case $l\geq2$ and $l+1\leq s\leq4l$ one can check that $\pi^{\infty}_s\left((F^{(l+1)}_t)_{t\in\boldsymbol{k}}\right)$ is a~one-parametr group of formal power series. We also checked\footnote{using {\em Computer Algebra System {\bf Maxima}}} that for $s\leq 15$ and $l\in\{2,3,4,5\}$ the family $\pi^{\infty}_s\left((F^{(l+1)}_t)_{t\in\boldsymbol{k}}\right)$ is a~one-parameter group of formal power series.
\end{remark}

In particular we get then

\begin{example}
The family
$$
\begin{array}{rcl}
{\mathcal{F}}&=&\displaystyle {\mathcal{L}}_2\overline{\ocirc}{\mathcal{S}}_1( {\mathrm{AJ}}_{15}(x^3))={\mathcal{L}}_2\overline{\ocirc}\left( \pi^{\infty}_{15}\left((F^{(3)}_{c_3})_{c_3\in\boldsymbol{k}}\right)\ocirc {\mathcal{S}}_1^{2}({\mathrm{AJ}}_{15}(x^3))\right)\\[2ex]
&=&\displaystyle  \left\{ c_1x+c_1c_3x^3+\frac{3}{2}c_1c_3^2x^5+ \frac{5}{2}c_1c_3^3x^7+ \right.\\[3ex]
& & \displaystyle  \qquad\quad \frac{35}{8}c_1c_3^4x^9+ \frac{63}{8}c_1c_3^5x^{11}+\frac{231}{16}c_1c_3^6x^{13} +   c_1c_{14}x^{14}+\\[3ex]
& & \displaystyle  \qquad\qquad\qquad \left.c_1\left(\frac{429}{16} c_3^7+c_{15}\right)x^{15}: c_1\in\{-1,1\}, c_3,c_{14},c_{15}\in\boldsymbol{k}\,\right\}
\end{array}
$$
is a non-commutative group of solutions of the Acz{\'e}l-Jabotinsky formal differential equation $({\mathrm{AJ}}_{15}(x^3))$ with the generator $H(x)=x^3$.
\end{example}

\begin{remark}
The coefficients $c_1c_{14}$ and $c_1\left(\frac{429}{16} c_3^7+c_{15}\right)$ for the powers $x^{14}$ and $x^{15}$ can formally be written as $\widetilde{c}_{14}$ and $\widetilde{c}_{15}$ with arbitrary values $\widetilde{c}_{14},\widetilde{c}_{15}\in\boldsymbol{k}$. We write them in the present form to show the structure of the group ${\mathcal{F}}$ as a product of groups.
\end{remark}

\subsection*{Statements and Declarations}

\subsubsection*{Funding}

The research of W. Jab{\l}o{\'n}ski was partially supported by the Faculty of Applied Mathematics AGH UST statutory tasks and dean grant within subsidy of Ministry of Science and Higher Education.

\subsubsection*{Competing Interests}

The authors declare they have no financial interests.

\subsubsection*{Data availability}

Not applicable.

\end{document}